\documentclass[11pt]{article}
\usepackage{epsfig}
\usepackage{amssymb,amsmath,amsthm,amscd}
\usepackage{latexsym}

\newtheorem{thm}{Theorem}[section]
\newtheorem{prop}[thm]{Proposition}

\newtheorem{lem}[thm]{Lemma}

\theoremstyle{definition}
\newtheorem{defn}[thm]{Definition}

\newcounter{labelflag} 

\newcommand{\Label}[1]{
                       \ifnum\thelabelflag=1
                          \ifmmode
                             \makebox[0in][l]{\qquad\fbox{\rm#1}}
                          \else
                             \marginpar{\vspace{0.7\baselineskip}
                                        \hspace{-1.1\textwidth}
                                        \fbox{\rm#1}}
                          \fi
                       \fi
                       \label{#1}
                      }
\newcommand{\be}{\begin{equation}}
\newcommand{\ee}{\end{equation}}

\newcommand{\vt}{{\tilde{v}}}
\newcommand{\ut}{{\tilde{u}}}
\newcommand{\h}{L^2(\R^n)}
\newcommand{\hone}{H^1(\R^n)}
\newcommand{\ii}{\int_{\R^n}}
\newcommand{\zto}{z_1(\theta_t \omega) }
\newcommand{\ztotwo}{z_2(\theta_t \omega) }
\newcommand{\rh }{\rho ({\frac {|x|^2}{k^2}})}
\newcommand{\rhp }{\rho^\prime ({\frac {|x|^2}{k^2}})}

 \newcommand{\R}{\mathbb{R}}

\begin{document}

\begin{titlepage}
\title{\Large\bf   Random Attractors for  the  Stochastic
FitzHugh-Nagumo System
on Unbounded Domains}
\vspace{7mm}

\author{
Bixiang Wang  \thanks {Supported in part by NSF  grant DMS-0703521}
\vspace{1mm}\\
Department of Mathematics, New Mexico Institute of Mining and
Technology \vspace{1mm}\\ Socorro,  NM~87801, USA \vspace{1mm}\\
Email: bwang@nmt.edu}
\date{}
\end{titlepage}

\maketitle

\medskip

\begin{abstract}
The existence of a random attractor for the 
stochastic FitzHugh-Nagumo system
defined on an unbounded domain is established.
The pullback
   asymptotic compactness of the  stochastic  system
   is proved  by uniform estimates on solutions 
   for large space and time variables. These estimates 
   are obtained by a cut-off  technique.
\end{abstract}



%

%
%

\baselineskip=1.5\baselineskip

\section{Introduction}
\setcounter{equation}{0}
In this paper, we investigate the asymptotic behavior of
solutions of the  following
  stochastic FitzHugh-Nagumo system   defined  on $\R^n$:
$$
du + (\lambda u - \Delta u + \alpha v )dt = (f(x,u) + g) dt +  \phi_1 dw_1,
\quad  x \in \R^n, \ \ t >0,
$$
$$
 dv + (\delta v - \beta u) dt =  h dt + \phi_2 d w_2,
 \quad  x \in \R^n, \ \ t >0,
$$
 where $\lambda, \alpha, \delta$ and $ \beta$  are    positive constants,
 $g \in L^2(\R^n) $ and $h \in \hone$  are  given,
 $\phi_1 \in H^2(\R^n) \cap W^{2, p}(\R^n)$
for some  $p\ge 2$,   $\phi_2 \in \hone$,
$f$ is a nonlinear  function satisfying  certain dissipative conditions,
$w_1$ and $w_2$   are independent two-sided real-valued
 Wiener processes on a probability
 space which will be specified later.

 It is known that the asymptotic  behavior of  a
  random system is 
 determined  
by a pullback random attractor. The concept of
random attractors was introduced in \cite{cra2, fla1}, 
which is  an analogue of
 global attractors for deterministic
dynamical systems as studied in  
   \cite{bab1, hal1, rob1, sel1, tem1}.
   When PDEs are defined in {\it  bounded} domains,
   the existence of random attractors   has been 
   investigated 
    by many  authors, see, e.g.,  \cite{arn1, car1, cra1, cra2, fla1}
and the references therein.
However, in the case of {\it unbounded} domains, 
the existence of random attractors is not well understood.
  It seems that  the work 
  \cite{bat2}
  by Bates, Lu and Wang is the only existence result 
  of   random attractors
  for PDEs   defined  on  unbounded domains, where  a random attractor 
   for   the  Reaction-Diffusion equation  on $\R^n$
   was established.
  In this paper, we will  study 
     such attractors  for the stochastic FitzHugh-Nagumo
system defined on unbounded domains.

Notice that 
 Sobolev embeddings are not  compact
 when domains are unbounded.   
This
introduces a major obstacle  for proving the existence of   attractors for PDEs on unbounded domains.
For some 
    deterministic PDEs,   such   difficulty  
     can be overcome
by  the
  energy equation approach that was   developed by  Ball in
\cite{bal1, bal2} and used in 
  \cite{ghi1,  gou1, ju1, moi2, moi1,  ros1, wanx}).
  Under certain circumstances, 
  the
   tail-estimates method   must be   used  to deal with
    the problem caused by the unboundedness of domains.
      This  approach  was  developed  in \cite{wan1}
      for deterministic parabolic equations 
   and used in
    \cite{ant1, ant2, arr1,   mor1, pri1, rod1, sta1, sun1}.
    Recently, the authors of \cite{bat2} extended the
    tail-estimates method to the stochastic parabolic  equations.
    In this  paper, we will use  the idea of uniform estimates 
    on the tails of solutions   
    to investigate the asymptotic behavior of the  stochastic FitzHugh-Nagumo system defined  on $\R^n$.

 This paper is organized  as follows. In the next section, we
 review the  pullback
 random attractors  theory  for random dynamical systems. In Section 3,
 we define  a continuous  random dynamical system for the
 stochastic FitzHugh-Nagumo  system  on $\R^n$.
 Then we derive the uniform estimates of solutions  in Section 4, which include
 the uniform estimates on the tails of solutions.
Finally,  we   establish the asymptotic compactness of
the random dynamical system and prove
  the existence of   a  pullback random attractor.

In the sequel,  we adopt  the following notations.  We denote by
$\| \cdot \|$ and $(\cdot, \cdot)$ the norm and the inner product
of
  $L^2(\R^n)$, respectively.    The
norm of  a given  Banach space $X$  is written as    $\|\cdot\|_{X}$.
We also use $\| \cdot\|_{p}$    to denote   the norm  of
$L^{p}(\R^n)$.  The letters $c$ and $c_i$ ($i=1, 2, \ldots$)
are  generic positive constants  which may change their  values from line to
line or even in the same line.

\section{Preliminaries}
\setcounter{equation}{0}

In this section, we recall some basic concepts
related to random attractors for stochastic dynamical
systems. The reader is referred to \cite{arn1, bat1, cra1, fla1} for more details.

Let  $(X, \| \cdot \|_X)$ be a   separable
Hilbert space with Borel $\sigma$-algebra $\mathcal{B}(X)$,
 and
$(\Omega, \mathcal{F}, P)$  be  a probability space.

\begin{defn}
$(\Omega, \mathcal{F}, P,  (\theta_t)_{t\in \R})$
is called a  metric   dynamical  system
if $\theta: \R \times \ \Omega \to \Omega$ is
$(\mathcal{B}(\R) \times \mathcal{F}, \mathcal{F})$-measurable,
$\theta_0$ is the identity on $\Omega$,
$\theta_{s+t} = \theta_t \circ  \theta_s$ for all
$s, t \in \R$ and $\theta_t P = P$ for all $t \in \R$.
 \end{defn}

\begin{defn}
\label{RDS}
A continuous random dynamical system (RDS)
on $ X  $ over  a metric  dynamical system
$(\Omega, \mathcal{F}, P,  (\theta_t)_{t\in \R})$
is  a mapping
  $$
\Phi: \R^+ \times \Omega \times X \to X \quad (t, \omega, x)
\mapsto \Phi(t, \omega, x),
$$
which is $(\mathcal{B}(\R^+) \times \mathcal{F} \times \mathcal{B}(X), \mathcal{B}(X))$-measurable and
satisfies, for $P$-a.e.  $\omega \in \Omega$,

(i) \  $\Phi(0, \omega, \cdot) $ is the identity on $X$;

(ii) \  $\Phi(t+s, \omega, \cdot) = \Phi(t, \theta_s \omega,
\cdot) \circ \Phi(s, \omega, \cdot)$ for all $t, s \in \R^+$;

(iii) \  $\Phi(t, \omega, \cdot): X \to  X$ is continuous for all
$t \in  \R^+$.
\end{defn}

Hereafter, we always assume that $\Phi$  is a continuous RDS on $X$
over $(\Omega, \mathcal{F}, P,  (\theta_t)_{t\in \R})$.

\begin{defn}
 A random  bounded set $\{B(\omega)\}_{\omega \in \Omega}$
 of  $  X$  is called  tempered
 with respect to $(\theta_t)_{t\in \R}$ if for $P$-a.e. $\omega \in \Omega$,
 $$ \lim_{t \to \infty} e^{- \beta t} d(B(\theta_{-t} \omega)) =0
 \quad \mbox{for all} \  \beta>0,
 $$
 where $d(B) =\sup_{x \in B} \| x \|_{X}$.
\end{defn}

\begin{defn}
Let $\mathcal{D}$ be a collection of  random  subsets of $X$.
Then  $\mathcal{D}$ is called inclusion-closed if
   $D=\{D(\omega)\}_{\omega \in \Omega} \in {\mathcal{D}}$
and  $\tilde{D}=\{\tilde{D}(\omega) \subseteq X:  \omega \in \Omega\} $
with
  $\tilde{D}(\omega) \subseteq D(\omega)$ for all $\omega \in \Omega$ imply
  that  $\tilde{D} \in {\mathcal{D}}$.
  \end{defn}

\begin{defn}
Let $\mathcal{D}$ be a collection of random subsets of $X$ and
$\{K(\omega)\}_{\omega \in \Omega} \in \mathcal{D}$. Then
$\{K(\omega)\}_{\omega \in \Omega} $ is called an absorbing set of
$\Phi$ in $\mathcal{D}$ if for every $B \in \mathcal{D}$ and
$P$-a.e. $\omega \in \Omega$, there exists $t_B(\omega)>0$ such
that
$$
\Phi(t, \theta_{-t} \omega, B(\theta_{-t} \omega)) \subseteq
K(\omega) \quad \mbox{for all} \ t \ge t_B(\omega).
$$
\end{defn}

\begin{defn}
Let $\mathcal{D}$ be a collection of random subsets of $X$. Then
$\Phi$ is said to be  $\mathcal{D}$-pullback asymptotically
compact in $X$ if  for $P$-a.e. $\omega \in \Omega$,
$\{\Phi(t_n, \theta_{-t_n} \omega,
x_n)\}_{n=1}^\infty$ has a convergent  subsequence  in $X$
whenever
  $t_n \to \infty$, and $ x_n\in   B(\theta_{-t_n}\omega)$   with
$\{B(\omega)\}_{\omega \in \Omega} \in \mathcal{D}$.
\end{defn}

\begin{defn}
Let $\mathcal{D}$ be a collection of random subsets of $X$
and $\{\mathcal{A}(\omega)\}_{\omega \in \Omega} \in  \mathcal{D}$.
Then   $\{\mathcal{A}(\omega)\}_{\omega \in \Omega} $
is called a   $\mathcal{D}$-random   attractor
(or $\mathcal{D}$-pullback attractor)  for
  $\Phi$
if the following  conditions are satisfied, for $P$-a.e. $\omega \in \Omega$,

(i) \  $\mathcal{A}(\omega)$ is compact,  and
$\omega \mapsto d(x, \mathcal{A}(\omega))$ is measurable for every
$x \in X$;

(ii) \ $\{\mathcal{A}(\omega)\}_{\omega \in \Omega}$ is invariant, that is,
$$ \Phi(t, \omega, \mathcal{A}(\omega)  )
= \mathcal{A}(\theta_t \omega), \ \  \forall \   t \ge 0;
$$

(iii) \ \ $\{\mathcal{A}(\omega)\}_{\omega \in \Omega}$
attracts  every  set  in $\mathcal{D}$,  that is, for every
 $B = \{B(\omega)\}_{\omega \in \Omega} \in \mathcal{D}$,
$$ \lim_{t \to  \infty} d (\Phi(t, \theta_{-t}\omega, B(\theta_{-t}\omega)), \mathcal{A}(\omega))=0,
$$
where $d$ is the Hausdorff semi-metric given by
$d(Y,Z) =
  \sup_{y \in Y }
\inf_{z\in  Z}  \| y-z\|_{X}
 $ for any $Y\subseteq X$ and $Z \subseteq X$.
\end{defn}

The following existence result on a    random attractor
for a  continuous  RDS
can be found in \cite{bat1,  fla1}.

\begin{prop}
\label{att} Let $\mathcal{D}$ be an inclusion-closed
 collection of random subsets of
$X$ and $\Phi$ a continuous RDS on $X$ over $(\Omega, \mathcal{F},
P,  (\theta_t)_{t\in \R})$. Suppose  that $\{K(\omega)\}_{\omega
\in \Omega} $ is a closed  absorbing set of  $\Phi$  in $\mathcal{D}$
and $\Phi$ is $\mathcal{D}$-pullback asymptotically compact in
$X$. Then $\Phi$ has a unique $\mathcal{D}$-random attractor
$\{\mathcal{A}(\omega)\}_{\omega \in \Omega}$ which is given by
$$\mathcal{A}(\omega) =  \bigcap_{\tau \ge 0} \  \overline{ \bigcup_{t \ge \tau} \Phi(t, \theta_{-t} \omega, K(\theta_{-t} \omega)) }.
$$
\end{prop}

 In this paper, we will take $\mathcal{D}$ as
 the collection of all tempered  random subsets of $\h\times \h$ and prove the
 stochastic FitzHugh-Nagumo system  on $\R^n$ has
a $\mathcal{D}$-random attractor.

\section{Stochastic FitzHugh-Nagumo system
on  $\R^n$  }
\setcounter{equation}{0}

In this section, we discuss the existence of a continuous random dynamical system
for the stochastic FitzHugh-Nagumo system  defined  on $\R^n$:
 \be
\label{rd1}
du + (\lambda u - \Delta u + \alpha v )dt = (f(x,u) + g) dt +  \phi_1 dw_1,
\quad  x \in \R^n, \ \ t >0,
 \ee
 \be
 \label{rd1_a1}
 dv + (\delta v - \beta u) dt =  h dt + \phi_2 d w_2,
 \quad  x \in \R^n, \ \ t >0,
 \ee
 with the initial conditions:
 \be
 \label{rd2}
 u(x, 0) =u_0(x),  \quad v(x,0) = v_0 (x), \quad x \in \R^n,
 \ee
 where $\lambda, \alpha, \delta$ and $ \beta$  are    positive constants,
 $g \in L^2(\R^n) $ and $h \in \hone$  are  given,
 $\phi_1 \in H^2(\R^n) \cap W^{2, p}(\R^n)$
for some  $p\ge 2$,  $  \phi_2   \in \hone$,
 $w_1$ and $w_2$   are independent two-sided real-valued
 Wiener processes on a probability
 space which will be specified below, and
   $f$ is a nonlinear  function   satisfying the conditions,
 $ \forall x \in  R^n$ and $ s \in R$,
\begin{equation}
\label{f1}
f(x, s) s \le -  \alpha_1 |s|^p + \psi_1(x),
\end{equation}
\begin{equation}
\label{f2}
|f(x, s) |   \le \alpha_2 |s|^{p-1} + \psi_2 (x),
\end{equation}
\begin{equation}
\label{f3}
{\frac {\partial f}{\partial s}} (x, s)   \le \beta,
\end{equation}
\begin{equation}
\label{f4}
| {\frac {\partial f}{\partial x}} (x, s) | \le  \psi_3(x),
\end{equation}
where $\alpha_1$, $\alpha_2$ and    $\beta$
 are positive constants,
$\psi_1 \in L^1(R^n) \cap L^\infty(R^n)$, and $\psi_2 \in L^2(R^n) \cap L^q(\R^n) $
with ${\frac 1q} + {\frac 1p} =1$, and $\psi_3 \in L^2(\R^n)$.

In the sequel, we consider the probability space
$(\Omega, \mathcal{F}, P)$ where
$$
\Omega = \{ \omega =(\omega_1, \omega_2)  \in C(\R, \R^2): \ \omega(0) =  0 \},
$$
$\mathcal{F}$ is the Borel $\sigma$-algebra induced by the compact-open topology
of $\Omega$, and $P$ the corresponding Wiener measure on
$(\Omega, \mathcal{F})$.  Then we have
$$
 (w_1(t, \omega),  w_2(t, \omega) ) = \omega (t), \quad  t \in \R.
$$
 Define the time shift by
$$ \theta_t \omega (\cdot) = \omega (\cdot +t) - \omega (t), \quad  \omega \in \Omega, \ \ t \in \R .
$$
Then $(\Omega, \mathcal{F}, P, (\theta_t)_{t\in \R})$
is a measurable dynamical  system.

Consider the stationary solutions of the one-dimensional
equations
\be
\label{y1}
dy_1 + \lambda y_1 dt = d w_1(t),
\ee
and
\be
\label{y2}
dy_2 + \delta  y_2 dt = d w_2(t).
\ee
The solutions to \eqref{y1} and \eqref{y2}  are  given by
$$
y_1 (\theta_t \omega_1)= -\lambda \int^0_{-\infty} e^{\lambda \tau}  ( \theta_t \omega_1 )
(\tau) d \tau, \quad t \in \R,
$$
and
$$
y_2 (\theta_t \omega_2)= -\delta \int^0_{-\infty} e^{\delta \tau}  ( \theta_t \omega_2 )
(\tau) d \tau, \quad t \in \R.
$$
Note that  there exists a $\theta_t$-invariant set $\tilde{\Omega}\subseteq \Omega$
of full $P$ measure  such that
  $y_j(\theta_t\omega_j)$ ($j=1, 2$) is
 continuous in $t$ for every $\omega \in \tilde{\Omega}$,
and
the random variable $|y_j(\omega_j)|$ is tempered
(see, e.g., \cite{bat1}).
 Therefore, it follows from Proposition 4.3.3 in \cite{arn1}
that there exists a tempered function $r(\omega)>0$ such that
\be
\label{z2}
\sum_{j=1}^2  \left ( |y_j(\omega_j)|^2 + |y_j(\omega_j)|^p \right )
\le r(\omega).
\ee
where $r(\omega)$ satisfies, for $P$-a.e. $\omega \in \Omega$,
\be
\label{z3}
r(\theta_t \omega ) \le e^{{\frac \eta{2}} |t|} r(\omega), \quad t\in \R,
\ee
where
  $\eta =\min \{\lambda, \delta \}$.
Then it follows from \eqref{z2}-\eqref{z3} that, for $P$-a.e. $\omega \in \Omega$,
\be
\label{zz}
\sum_{j=1}^2  \left ( |y_j(\theta_t \omega_j)|^2 + |y_j(\theta_t \omega_j)|^p \right )
\le e^{{\frac \eta{2}} |t|} r(\omega), \quad t\in \R.
\ee
Putting $z_j (\theta_t \omega) =   \phi_j  y_j(\theta_t \omega_j ) $,  by \eqref{y1}
and \eqref{y2}
we have
$$ dz_1 + \lambda z_1 dt =  \phi_1 d w_1 ,$$
and
$$ dz_2 + \delta z_2  dt =  \phi_2 d w_2 .$$
Let $\ut = u(t) - z_1(\theta_t \omega)$
and $\vt = v(t) - z_2(\theta_t \omega)$,
where $(u, v)$  satisfies \eqref{rd1}-\eqref{rd2}.
Then for  $(\ut, \vt)$, we have
\be
\label{u1}
{\frac {d\ut}{dt}} + \lambda \ut -\Delta \ut
+ \alpha \vt  = f(x, \ut  + z_1 (\theta_t \omega) ) + g  + \Delta z_1 (\theta_t \omega )
-\alpha z_2 (\theta_t \omega ),
\ee
and
\be
\label{v1}
{\frac {d\vt}{dt}} + \delta \vt
-\beta \ut   = h   + \beta z_1 (\theta_t \omega ).
\ee
By a Galerkin method as in \cite{mar1}, it  can be proved  that if $f$ satisfies \eqref{f1}-\eqref{f4}, then
for $P$-a.e. $\omega \in \Omega $ and for all $(\ut_0, \vt_0)  \in L^2(\R^n) \times \h$,
  system \eqref{u1}-\eqref{v1}
has   a unique solution
$(  \ut (\cdot, \omega, \ut_0), \vt (\cdot, \omega, \vt_0)
) \in C([0, \infty), \h \times \h)  $
  with $\ut(0, \omega, \ut_0) = \ut_0$  and
  $\vt(0, \omega, \vt_0) = \vt_0$.
  Further,  the solution $(  \ut (t, \omega, \ut_0), \vt (t, \omega, \vt_0) )$
  is continuous    with respect to $(\ut_0, \vt_0 ) $ in $L^2(\R^n) \times \h$
for all $t \ge 0$.  Throughout this paper, we always write
\be
\label{uv}
 u(t, \omega, u_0)=
\ut (t, \omega,  u_0 - z_1(\omega)) + z_1(\theta_t \omega), \
 \quad v(t, \omega, v_0)= \vt (t, \omega,  v_0 -
z_2(\omega)) + z_2(\theta_t \omega). \ee Then $(u, v)$ is  a
solution of problem \eqref{rd1}-\eqref{rd2} in some sense. We now
define  a mapping  $\Phi: \R^+ \times \Omega \times ( L^2(\R^n)
\times \h ) \to L^2(\R^n) \times \h$ by \be \label{phi}
 \Phi (t, \omega, (u_0, v_0) ) =  ( u (t, \omega, u_0), v(t, \omega, v_0) )  ,
 \quad \forall \ (t, \omega, (u_0, v_0)) \in
\R^+ \times \Omega \times ( L^2(\R^n)  \times \h ). \ee
Note that
$\Phi$   satisfies conditions (i), (ii) and (iii) in Definition
\ref{RDS}. Therefore, $\Phi$ is a continuous  random dynamical
system
 associated  with  the stochastic FitzHugh-Nagumo
 system  on $\R^n$.  In the next two sections, we establish   uniform estimates
  for $\Phi$
and prove the existence
  of a random attractor.

\section{Uniform  estimates of solutions  }
\setcounter{equation}{0}

In this section, we
 derive uniform estimates on the  solutions of the  stochastic
FitzHugh-Nagumo system  defined on $\R^n$
when $t \to \infty$.   These estimates  are necessary
 for proving  the existence of bounded absorbing sets
 and the asymptotic compactness of
  the random dynamical system associated  with the system.
  Particularly, we
  will  show that   the tails of the solutions
 for large space variables are uniformly small
 when time is sufficiently  large.

 From now on, we always assume that $\mathcal{D}$ is
 the collection of all tempered subsets of  $L^2(\R^n) \times \h$
with respect to  $(\Omega, \mathcal{F}, P, (\theta_t)_{t \in
\R})$. The next lemma shows that $\Phi$ has an absorbing  set in
$\mathcal{D}$.

\begin{lem}
\label{lem41} Assume that $g, h  \in L^2(\R^n)$ and
\eqref{f1}-\eqref{f4} hold.  Then there exists
$\{K(\omega)\}_{\omega \in \Omega}\in \mathcal{D}$ such that
$\{K(\omega)\}_{\omega\in \Omega} $ is a closed  absorbing  set for
$\Phi$ in $\mathcal{D}$, that is, for any $B=
\{B(\omega)\}_{\omega \in \Omega} \in \mathcal{D}$ and $P$-a.e.
$\omega \in \Omega$, there is $T_B(\omega)>0$ such that
$$
\Phi(t, \theta_{-t} \omega, B(\theta_{-t} \omega) ) \subseteq
K(\omega) \quad \mbox{for all}  \   t \ge T_B(\omega).
$$
\end{lem}

\begin{proof}
Taking the inner product of \eqref{u1} with $\beta \ut$ in $L^2(\R^n)$ we find that
$$
{\frac 12} \beta {\frac d{dt}} \|\ut \|^2 + \lambda \beta  \| \ut \|^2 + \beta  \| \nabla \ut  \|^2
+ \alpha \beta (\vt, \ut)
$$
\be
\label{p41_1}
=
 \beta ( f(x, \ut + z_1 (\theta_t \omega ) ) , \  \ut ) + \beta (g, \ut)
+ \beta (  \Delta z_1 (\theta_t \omega ),  \ut)
-\alpha \beta (z_2 (\theta_t \omega), \ut).
 \ee
 Taking the inner product of \eqref{v1} with $\alpha  \vt$ in $L^2(\R^n)$ we find that
\be
\label{p41_1_a1}
{\frac 12} \alpha  {\frac d{dt}} \|\vt \|^2 + \alpha \delta   \| \vt \|^2
-  \alpha \beta (\vt, \ut)
= \alpha (h, \vt) + \alpha \beta (
  z_1 (\theta_t \omega ),  \vt) .
  \ee
  Adding \eqref{p41_1} and \eqref{p41_1_a1}, we obtain that
  $$
  {\frac 12} {\frac d{dt}} \left (
   \beta \| \ut \|^2 + \alpha \| \vt \|^2
  \right )
  + \lambda \beta \| \ut \|^2
  + \alpha \delta \| \vt \|^2
  + \beta \| \nabla \ut \|^2
  $$
$$
=
 \beta ( f(x, \ut + z_1 (\theta_t \omega ) ) , \  \ut ) + \beta (g, \ut)
+ \beta (  \Delta z_1 (\theta_t \omega ),  \ut)
$$
  \be
\label{p41_1_a2}
-\alpha \beta (z_2 (\theta_t \omega), \ut) +
\alpha (h, \vt) + \alpha \beta (
  z_1 (\theta_t \omega ),  \vt).
 \ee
  We now estimate  every term on the right-hand side of
  \eqref{p41_1_a2}.
  For the nonlinear term, by \eqref{f1}-\eqref{f2} we obtain that
$$
( f(x, \ut + z_1 (\theta_t \omega ) ) , \   \ut )
= ( f(x, \ut+ z_1(\theta_t \omega ) ) , \  \ut + z_1 (\theta_t \omega )  )
- ( f(x, \ut + z_1(\theta_t \omega ) ) , \  z_1 (\theta_t \omega ) )
$$
$$
\le -\alpha_1 \ii | u |^p dx + \ii \psi_1(x) dx
- \ii f(x, u)\;  \zto dx
$$
$$
\le -\alpha_1 \ii | u |^p dx + \ii \psi_1(x) dx
+ \alpha_2  \ii |u|^{p-1}| \; | \zto | dx + \ii |\psi_2|\;  |\zto| dx
$$
$$
\le -\alpha_1   \| u \|_p ^p  +  \| \psi_1\|_1
+ {\frac 12} \alpha_1    \|u \|_p^{p}
+ c_1  \|\zto \|^p _p
+ {\frac 12} \| \psi_2 \|^2 + {\frac 12} \| \zto \|^2
$$
\be
\label{p41_2}
\le - {\frac 12} \alpha_1 \| u \|^p_p
+ c_2 ( \|\zto \|^p _p  +  \|\zto \|^2 ) +c_3.
\ee
The second term on the right-hand side of \eqref{p41_1_a2}
is bounded by
\be
\label{p41_3}
 \beta |  (g,  \ut ) | \le
\beta \| g \| \| \ut \|
\le
{\frac 14} \lambda \beta  \| \ut \|^2 + {\frac \beta{ \lambda  }} \| g \|^2.
\ee
For the third term on the right-hand side of  \eqref{p41_1_a2}, we have
\be
\label{p41_3_a1}
\beta |(\nabla \zto, \nabla \ut ) |
\le
{\frac 12} \beta \| \nabla \zto \|^2 + {\frac 12} \beta \| \nabla \ut \|^2
\ee
Similarly, the remaining terms on the right-hand side of \eqref{p41_1_a2}
is bounded  by
$$
 \alpha \beta  | (z_2 (\theta_t \omega), \ut) | +
\alpha  | (h, \vt)|  + \alpha \beta |(
  z_1 (\theta_t \omega ),  \vt) |
$$
$$
\le {\frac 14} \lambda \beta \| \ut \|^2
+ {\frac 1\lambda}\alpha^2 \beta \| \ztotwo \|^2
+ {\frac 14} \alpha \delta \| \vt \|^2
+ {\frac \alpha\delta} \| h \|^2
+   {\frac 14} \alpha \delta \| \vt \|^2
+ {\frac {\alpha \beta^2}\delta} \| \zto \|^2.
$$
\be
\label{p41_3_a2}
\le {\frac 14} \lambda \beta \| \ut \|^2
+ {\frac 1\lambda}\alpha^2 \beta \| \ztotwo \|^2
+ {\frac 12} \alpha \delta \| \vt \|^2
+ {\frac \alpha\delta} \| h \|^2
+ {\frac {\alpha \beta^2}\delta} \| \zto \|^2.
\ee
Then it follows from \eqref{p41_1_a2}-\eqref{p41_3_a2} that
$$
  {\frac d{dt}} \left (
   \beta \| \ut \|^2 + \alpha \| \vt \|^2
  \right )
  + \lambda \beta \| \ut \|^2
  + \alpha \delta \| \vt \|^2
  + \beta \| \nabla \ut \|^2
  +
  \alpha_1 \| u \|^p_p
$$
\be
\label{p41_4}
\le
 c_4  \left ( \| \zto \|^p _p +  \| \zto \|^2 + \| \nabla \zto \|^2
+ \| \ztotwo \|^2 \right )
+ c_5.
\ee
Note that $\zto =  \phi_1   y_1 (\theta_t \omega_1) $, $\ztotwo =
\phi_2   y_2 (\theta_t \omega_2) $, $  \phi_2 \in H^1(\R^n)$  and
  $\phi_1  \in H^2(\R^n) \cap W^{2,p}(\R^n)$.
Therefore,  the right-hand  of \eqref{p41_4} is bounded by
\be
\label{p41_5}
c_6 \sum_{j=1}^2( | y_j (\theta_t \omega_j)  |^p + | y_j (\theta_t \omega_j)  |^2 ) + c_7
=p_1(\theta_t \omega) +c_7.
\ee
 By  \eqref{zz}, we find that for $P$-a.e. $\omega \in \Omega$,
 \be
 \label{p41_6a1}
 p_1(\theta_\tau \omega)
\le  c_6  e^{{\frac 12} \eta |\tau|} r(\omega) , \quad \forall \ \tau \in \R.
\ee
It follows from \eqref{p41_4}-\eqref{p41_5} that,  for all $t \ge 0$,
\be
\label{p41_6}
  {\frac d{dt}} \left (
   \beta \| \ut \|^2 + \alpha \| \vt \|^2
  \right )
  + \lambda \beta \| \ut \|^2
  + \alpha \delta \| \vt \|^2
  + \beta \| \nabla \ut \|^2
  +
  \alpha_1 \| u \|^p_p
 \le p_1(\theta_t \omega)  + c_7,
\ee
which implies that, for all $t \ge 0$,
$$
    {\frac d{dt}} \left (
   \beta \| \ut \|^2 + \alpha \| \vt \|^2
  \right )
  + \lambda \beta \| \ut \|^2
  + \alpha \delta \| \vt \|^2
 \le p_1(\theta_t \omega)  + c_7.
$$
Note that  $\eta =\min \{\lambda, \delta \}$.  We get that, for all $t \ge  0$,
\be
\label{p41_7}
 {\frac d{dt}} \left (
   \beta \| \ut \|^2 + \alpha \| \vt \|^2
  \right )
  + \eta \left (  \beta \| \ut \|^2
  + \alpha   \| \vt \|^2 \right )
 \le p_1(\theta_t \omega)  + c_7.
 \ee
By  Gronwall's lemma, we find that, for all $t \ge 0$,
$$
 \beta \| \ut (t, \omega, \ut_0(\omega) ) \|^2 +
 \alpha \| \vt (t, \omega, \vt_0(\omega) ) \|^2
 $$
 \be
\label{p41_8}
\le e^{-\eta  t} \left  ( \beta  \| \ut_0 (\omega) \|^2 +  \alpha  \| \vt_0 (\omega) \|^2 \right )
+ \int_0^t e^{\eta(\tau -t)}  p_1(\theta_\tau \omega ) d\tau
+ {\frac {c_7}\eta}.
\ee
By replacing $\omega$ by $\theta_{-t} \omega$, we get from
\eqref{p41_8} and \eqref{p41_6a1}  that, for all $t \ge 0$,
$$
 \beta \| \ut (t, \theta_{-t} \omega, \ut_0(\theta_{-t}\omega) ) \|^2 +
 \alpha \| \vt (t, \theta_{-t}\omega, \vt_0(\theta_{-t}\omega) ) \|^2
 $$
$$
\le e^{-\eta  t} \left  ( \beta  \| \ut_0 (\theta_{-t}\omega) \|^2
+  \alpha  \| \vt_0 (\theta_{-t}\omega) \|^2 \right )
+ \int_0^t e^{\eta(\tau -t)}  p_1(\theta_{\tau-t} \omega ) d\tau
+ {\frac {c_7}\eta}.
$$
$$
\le e^{-\eta  t} \left  ( \beta  \| \ut_0 (\theta_{-t}\omega) \|^2
+  \alpha  \| \vt_0 (\theta_{-t}\omega) \|^2 \right )
+ \int^0_{-t } e^{\eta\tau }  p_1(\theta_{\tau} \omega ) d\tau
+ {\frac {c_7}\eta}.
$$
$$
\le e^{-\eta  t} \left  ( \beta  \| \ut_0 (\theta_{-t}\omega) \|^2
+  \alpha  \| \vt_0 (\theta_{-t}\omega) \|^2 \right )
+ c_6  \int^0_{-t } e^{{\frac 12} \eta\tau }  r( \omega ) d\tau
+ {\frac {c_7}\eta}.
$$
  \be
\label{p41_9}
\le e^{-\eta  t} \left  ( \beta  \| \ut_0 (\theta_{-t}\omega) \|^2
+  \alpha  \| \vt_0 (\theta_{-t}\omega) \|^2 \right )
+ {\frac {2c_6}\eta} r(\omega)
+ {\frac {c_7}\eta}.
\ee
By \eqref{uv} and \eqref{phi} we have
$$
\Phi(t, \omega, (u_0, v_0)(\omega) ) =  \left ( \ut (t, \omega,
u_0(\omega) - z_1(\omega) ) + z_1 (\theta_t \omega ) , \quad
 \vt (t, \omega, v_0(\omega) - z_2(\omega) ) + z_2 (\theta_t \omega )
\right ).
$$
So by \eqref{p41_9} we get that, for all $t \ge 0$,
$$
\|  \Phi(t, \theta_{-t} \omega, (u_0, v_0) (\theta_{-t} \omega) )
\|^2
$$
$$
= \| \ut (t,  \theta_{-t} \omega, u_0(\theta_{-t} \omega) - z_1(\theta_{-t} \omega) )
+ z_1 (  \omega )  \|^2
$$
$$
+
\| \vt (t,  \theta_{-t} \omega, v_0(\theta_{-t} \omega) - z_2(\theta_{-t} \omega) )
+ z_2 (  \omega )  \|^2
$$
$$
\le  2  \| \ut (t,  \theta_{-t} \omega, u_0(\theta_{-t} \omega) - z_1(\theta_{-t} \omega) ) \|^2
+ 2  \| z_1 (  \omega )  \|^2
$$
$$
+
 2 \| \vt (t,  \theta_{-t} \omega, v_0(\theta_{-t} \omega) - z_2(\theta_{-t} \omega) )  \|^2
+  2  \| z_2 (  \omega )  \|^2
$$
$$
 \le 2 \left ( {\frac 1\alpha} + {\frac 1\beta} \right )
 e^{- \eta t} \left (
  \beta \|   u_0(\theta_{-t} \omega) - z_1(\theta_{-t} \omega) \|^2
  + \alpha \|   v_0(\theta_{-t} \omega) - z_2(\theta_{-t} \omega) \|^2
 \right )
 $$
 $$
 + 2  \| z_1 (  \omega )  \|^2 + 2  \| z_2 (  \omega )  \|^2
 + c_8 r(\omega) + c_9
$$
$$
 \le c_{10}
 e^{- \eta t} \left (
  \| u_0(\theta_{-t} \omega) \|^2
  +   \|   v_0(\theta_{-t} \omega) \|^2 + \|  z_1(\theta_{-t} \omega) \|^2
+ \|  z_2(\theta_{-t} \omega) \|^2
 \right )
 $$
\be
\label{p41_20}
 + 2  \| z_1 (  \omega )  \|^2 + 2  \| z_2 (  \omega )  \|^2
 + c_8 r(\omega) + c_9
 \ee
Note that   $\| z_1 (\omega)\|^2$,   $\| z_2 (\omega)\|^2$
and
$\{B(\omega)\}_{\omega \in \Omega} \in \mathcal{D}$  are tempered.
   Therefore,
for  $(u_0, v_0)  (\theta_{-t} \omega ) \in B(\theta_{-t} \omega )$, then there is
 $T_B(\omega)>0$ such that for all $t \ge T_B(\omega)$,
 $$
 c_{10}
 e^{- \eta t} \left (
  \| u_0(\theta_{-t} \omega) \|^2
  +   \|   v_0(\theta_{-t} \omega) \|^2 + \|  z_1(\theta_{-t} \omega) \|^2
+ \|  z_2(\theta_{-t} \omega) \|^2
 \right )
   \le    c_8  r(\omega)
+ c_9,
 $$
 which along with \eqref{p41_20} shows that, for all $t \ge T_B(\omega)$,
\be \label{p41_21} \|  \Phi(t, \theta_{-t} \omega, (u_0, v_0)
(\theta_{-t} \omega) ) \|^2 \le 2 \left (
    \| z_1 (\omega ) \|^2  +  \| z_2 (\omega ) \|^2  +
 c_8  r(\omega) +c_9
\right ).
\ee
Given $\omega \in \Omega$, denote by
$$
K(\omega) = \{ (u,v)  \in \h \times \h : \ \| u \|^2  + \| v \|^2  \le
2 \left (
    \| z_1 (\omega ) \|^2  +  \| z_2 (\omega ) \|^2  +
 c_8  r(\omega) +c_9
\right )  \}.
$$
Then $\{K(\omega)\}_{\omega \in \Omega} \in \mathcal{D}$. Further,
\eqref{p41_21} shows  that $\{K(\omega)\}_{\omega \in \Omega}$ is
an absorbing set  for $\Phi$ in $\mathcal{D}$, which completes the
proof.
\end{proof}

\begin{lem}
\label{lem41a}
Assume that $g, h \in L^2(\R^n)$ and \eqref{f1}-\eqref{f4}
hold.   Let   $B=\{B(\omega)\}_{\omega \in \Omega}\in \mathcal{D}$.
Then for every $T_1\ge 0$ and  $P$-a.e. $\omega \in \Omega$,
   $u(t, \omega, u_0(\omega))$
  and
$\ut(t, \omega, \ut_0(\omega))$   satisfy,  for all $t \ge T_1$,
\be
\label{lem41a_1}
\int_{T_1}^{t }
 e^{\eta (s-t)}
\|  u (s, \theta_{-t} \omega, u_0(\theta_{-t} \omega)  )\|^p_p ds
\le  c  e^{-\eta t}
\left (
 \| \ut_0(\theta_{-t} \omega) \|^2 +
  \| \vt_0(\theta_{-t} \omega) \|^2 \right )
 + c(1 + r(\omega) ),
\ee
\be
\label{lem41a_2}
\int_{T_1}^{t }
 e^{\eta (s-t)}
\| \nabla \ut(s, \theta_{-t} \omega, \ut_0(\theta_{-t} \omega)  )\|^2 ds
\le
 c  e^{-\eta t}
\left (
 \| \ut_0(\theta_{-t} \omega) \|^2 +
  \| \vt_0(\theta_{-t} \omega) \|^2 \right )
 + c(1 + r(\omega) ),
\ee
where $c$ is a positive deterministic constant independent of $T_1$,
 and
  $r(\omega)$ is the  tempered function   in \eqref{z2}.
\end{lem}

\begin{proof}
First replacing $t$ by $T_1$ and  then replacing $\omega$ by $\theta_{-t} \omega$
in \eqref{p41_8}, we find that
$$
 \beta \| \ut (T_1, \theta_{-t}\omega, \ut_0(\theta_{-t}\omega) ) \|^2 +
 \alpha \| \vt (T_1, \theta_{-t}\omega, \vt_0(\theta_{-t}\omega) ) \|^2
 $$
$$
\le e^{-\eta  T_1 } \left  ( \beta  \| \ut_0 (\theta_{-t}\omega) \|^2
+  \alpha  \| \vt_0 (\theta_{-t}\omega) \|^2 \right )
+ \int_0^{T_1} e^{\eta(\tau -T_1)}  p_1(\theta_{\tau-t} \omega ) d\tau
+ c
$$
Multiply the above by $e^{\eta (T_1-t)}$ and then simplify to get that
$$e^{\eta (T_1-t)} \left (
 \beta \| \ut (T_1, \theta_{-t}\omega, \ut_0(\theta_{-t}\omega) ) \|^2 +
 \alpha \| \vt (T_1, \theta_{-t}\omega, \vt_0(\theta_{-t}\omega) ) \|^2 \right )
 $$
 \be
\label{p41a_1}
\le e^{-\eta  t } \left  ( \beta  \| \ut_0 (\theta_{-t}\omega) \|^2
+  \alpha  \| \vt_0 (\theta_{-t}\omega) \|^2 \right )
+ \int_0^{T_1} e^{\eta(\tau -t)}  p_1(\theta_{\tau-t} \omega ) d\tau
+ ce^{\eta (T_1-t)}.
 \ee
By \eqref{p41_6a1}, the second term on the right-hand side of
\eqref{p41a_1} satisfies
$$
\int_0^{T_1} e^{\eta(\tau -t)}  p_1(\theta_{\tau-t} \omega ) d\tau
=\int_{-t}^{T_1 -t} e^{\eta \tau} p_1(\theta_\tau \omega )  d \tau
$$
\be
\label{p41a_2}
\le c_6 r(\omega) \int_{-t}^{T_1 -t} e^{ {\frac 12} \eta \tau}  d \tau
\le {\frac 2\eta} c_6 r(\omega) e^{ {\frac 12} \eta (T_1 -t)}.
\ee
From \eqref{p41a_1}-\eqref{p41a_2} it follows  that
$$e^{\eta (T_1-t)} \left (
 \beta \| \ut (T_1, \theta_{-t}\omega, \ut_0(\theta_{-t}\omega) ) \|^2 +
 \alpha \| \vt (T_1, \theta_{-t}\omega, \vt_0(\theta_{-t}\omega) ) \|^2 \right )
 $$
\be
\label{p41a_3}
\le e^{-\eta  t } \left  ( \beta  \| \ut_0 (\theta_{-t}\omega) \|^2
+  \alpha  \| \vt_0 (\theta_{-t}\omega) \|^2 \right )
+  {\frac 2\eta} c_6 r(\omega) e^{ {\frac 12} \eta (T_1 -t)}
+ ce^{\eta (T_1-t)}.
 \ee
 Note that \eqref{p41_6} implies, for all $t \ge 0$,
$$
  {\frac d{dt}} \left (
   \beta \| \ut \|^2 + \alpha \| \vt \|^2
  \right )
  + \eta \left (  \beta \| \ut \|^2
  + \alpha  \| \vt \|^2 \right )
  + \beta \| \nabla \ut \|^2
  +
  \alpha_1 \| u \|^p_p
 \le p_1(\theta_t \omega)  + c .
$$
Multiplying the above by $e^{\eta t}$ and then integrating
over $(T_1, t)$, we get that, for all $ t\ge T_1$,
$$
\beta \| \ut (t, \omega, \ut_0(\omega)  )\|^2
+
\alpha \| \vt (t, \omega, \vt_0(\omega)  )\|^2
$$
$$
+ \int_{T_1}^t e^{\eta (s-t)} \| \nabla \ut(s, \omega, \ut_0(\omega) ) \|^2 ds
+
\alpha_1 \int_{T_1}^{t }
 e^{\eta (s-t)}
\|  u (s,  \omega, u_0( \omega)  )\|^p_p ds
$$
$$
\le e^{\eta (T_1-t)}  \left (
 \beta \| \ut (T_1 , \omega, \ut_0(\omega)  )\|^2
+
\alpha \| \vt (T_1, \omega, \vt_0(\omega)  )\|^2
\right )
$$
\be
\label{p41a_4}
+  \int_{T_1}^t e^{\eta (s-t)} p_1 (\theta_s \omega) ds
+c  \int_{T_1}^t e^{\eta (s-t)} ds.
\ee
Dropping the first  two   terms
on the left-hand side of \eqref{p41a_4} and
 replacing $\omega$ by $\theta_{-t} \omega$, we obtain   that, for all
$t \ge T_1$,
$$
  \int_{T_1}^t e^{\eta (s-t)}
\| \nabla \ut(s, \theta_{-t}\omega, \ut_0(\theta_{-t}\omega) ) \|^2 ds
+
\alpha_1 \int_{T_1}^{t }
 e^{\eta (s-t)}
\|  u (s,  \theta_{-t}\omega, u_0(\theta_{-t} \omega)  )\|^p_p ds
$$
$$
\le e^{\eta (T_1-t)}  \left (
 \beta \| \ut (T_1 , \theta_{-t}\omega, \ut_0(\theta_{-t}\omega)  )\|^2
+
\alpha \| \vt (T_1, \theta_{-t}\omega, \vt_0(\theta_{-t}\omega)  )\|^2
\right )
$$
$$
+  \int_{T_1}^t e^{\eta (s-t)} p_1 (\theta_{s-t} \omega) ds
+c  \int_{T_1}^t e^{\eta (s-t)} ds.
$$
$$
\le e^{\eta (T_1-t)}  \left (
 \beta \| \ut (T_1 , \theta_{-t}\omega, \ut_0(\theta_{-t}\omega)  )\|^2
+
\alpha \| \vt (T_1, \theta_{-t}\omega, \vt_0(\theta_{-t}\omega)  )\|^2
\right )
$$
\be
\label{p41a_5}
+  \int_{T_1-t}^0 e^{\eta \tau} p_1 (\theta_{\tau} \omega) d\tau
+ {\frac c\eta}
\ee
By \eqref{p41_6a1}, the second term on the right-hand
side of \eqref{p41a_5}
satisfies, for $t \ge T_1$,
\be
\label{p41a_6}
\int_{T_1-t}^0 e^{\eta \tau } p_1 (\theta_{\tau} \omega) d \tau
\le c_6 r(\omega)
\int_{T_1-t}^0 e^{{\frac 12} \eta \tau } d \tau
\le {\frac 2\eta} c_6 r(\omega).
\ee
Then, using \eqref{p41a_3} and \eqref{p41a_6}, it follows from
\eqref{p41a_5} that
$$
  \int_{T_1}^t e^{\eta (s-t)}
\| \nabla \ut(s, \theta_{-t}\omega, \ut_0(\theta_{-t}\omega) ) \|^2 ds
+
\alpha_1 \int_{T_1}^{t }
 e^{\eta (s-t)}
\|  u (s,  \theta_{-t}\omega, u_0(\theta_{-t} \omega)  )\|^p_p ds
$$
$$
 \le e^{-\eta  t} \left (
\beta \| \ut _0 (\theta_{-t} \omega ) \|^2
+ \alpha \| \vt _0 (\theta_{-t} \omega ) \|^2 \right )
 + c(1 + r(\omega) ).
 $$
 This completes the proof.
\end{proof}

As a special case of Lemma \ref{lem41a},  we have the following
uniform estimates.

\begin{lem}
\label{lem42}
Assume that $g, h \in L^2(\R^n)$ and \eqref{f1}-\eqref{f4}
hold.   Let   $B= \{B(\omega)\}_{\omega \in \Omega}\in \mathcal{D}$.
Then for $P$-a.e. $\omega \in \Omega$, there exists    $T_B(\omega)>0$
such that  for all $t \ge T_B(\omega)$,
$$
\int_t^{t+1} \|  u(s, \theta_{-t-1} \omega,
u_0(\theta_{-t-1 } \omega)  )\|^p_p ds
\le c (1+ r  (\omega) ),
$$
$$
\int_t^{t+1} \| \nabla \ut (s, \theta_{-t-1} \omega,
\ut_0(\theta_{-t-1 } \omega)  )\|^2 ds
\le c (1+ r  (\omega) ),
$$
where $c$ is a positive deterministic constant and
  $r(\omega)$ is the  tempered function   in \eqref{z2}.
\end{lem}

\begin{proof}
First replacing $t$ by $t+1$ and then
replacing $T_1$ by $t$ in \eqref{lem41a_2}, we find that
$$
\int_{t }^{t+1 }
 e^{\eta (s-t-1)}
\| \nabla \ut(s, \theta_{-t-1} \omega, \ut_0(\theta_{-t-1} \omega)  )\|^2 ds
$$
\be
\label{p42_1}
\le c  e^{-\eta (t+1)} \left (
 \| \ut _0(\theta_{-t-1} \omega) \|^2
+  \| \vt _0(\theta_{-t-1} \omega) \|^2 \right )
 + c(1 + r(\omega) ).
\ee
Note that $  e^{\eta (s-t-1)} \ge e^{-\eta}$ for $s \in [t, t+1]$. Hence
from \eqref{p42_1} we get that
$$
e^{-\eta}
\int_{t }^{t+1 }
\| \nabla \ut(s, \theta_{-t-1} \omega, \ut_0(\theta_{-t-1} \omega)  )\|^2 ds
$$
$$
\le c  e^{-\eta (t+1)} \left (
 \| \ut _0(\theta_{-t-1} \omega) \|^2
+  \| \vt _0(\theta_{-t-1} \omega) \|^2 \right )
 + c(1 + r(\omega) ).
$$
$$
\le c e^{-\eta (t+1)}
  \left ( \| u_0 ( \theta_{-t-1} \omega) \|^2
  + \| v_0 ( \theta_{-t-1} \omega) \|^2 \right )
 $$
   \be
\label{p42_2}
+ c e^{-\eta (t+1)} \left ( \| z_1 (\theta_{-t-1} \omega) \|^2
 +  \| z_2 (\theta_{-t-1} \omega) \|^2
\right )
+c( 1+ r(\omega)
 ).
\ee
Since $\|u_0(\omega)\|^2$,
$\|v_0(\omega)\|^2$,
$\|z_1(\omega)\|^2$
 and $\|z_2(\omega)\|^2$
are tempered,  there is $T _B(\omega)>0$
such that for all $t \ge T _B(\omega)$,
$$
c e^{-\eta (t+1)}
  \left ( \| u_0 ( \theta_{-t-1} \omega) \|^2
  + \| v_0 ( \theta_{-t-1} \omega) \|^2
+   \| z_1 (\theta_{-t-1} \omega) \|^2
 +  \| z_2 (\theta_{-t-1} \omega) \|^2
\right )
\le c( 1+ r(\omega)),
$$
which along with \eqref{p42_2} shows that, for all $t \ge T_B(\omega)$,
\be
\label{p42_3}
\int_{t }^{t+1 }
\| \nabla \ut(s, \theta_{-t-1} \omega, \ut_0(\theta_{-t-1} \omega)  )\|^2 ds
  \le 2 e^\eta (1 + r(\omega) ).
\ee
Using \eqref{lem41a_1} and repeating the above process, we can also find that,
for  $t \ge T_B(\omega)$,
\be
\label{p42_4}
   \int_{t }^{t+1 }
\|  u(s, \theta_{-t-1} \omega, u_0(\theta_{-t-1} \omega)  )\|^p_p ds
\le 2 e^\eta (1 + r(\omega) ).
\ee
Then the lemma follows from \eqref{p42_3}-\eqref{p42_4}.
\end{proof}

\begin{lem}
\label{lem42a}
Assume that $g, h \in L^2(\R^n)$ and \eqref{f1}-\eqref{f4}
hold.   Let   $B=\{B(\omega)\}_{\omega \in \Omega}\in \mathcal{D}$.
Then for $P$-a.e. $\omega \in \Omega$, there exists    $T_B(\omega)>0$
such that   for all $t \ge T_B(\omega)$,
$$
\int_t^{t+1} \| \nabla u(s, \theta_{-t-1} \omega, u_0(\theta_{-t-1 } \omega) \|^2 ds
\le c (1+ r  (\omega) ),
$$
where $c$ is a positive deterministic constant and
  $r(\omega)$ is the  tempered function   in \eqref{z2}.
\end{lem}

\begin{proof}
Let $T_B(\omega)$  be the positive constant in Lemma \ref{lem42},
take $t \ge T_B(\omega)$ and $s \in (t, t+1)$.
By \eqref{uv} we find that
$$
 \| \nabla u  (s, \theta_{-t-1} \omega, u_0(\theta_{-t-1 } \omega) \|^2
 =
  \| \nabla \ut  (s, \theta_{-t-1} \omega, \ut_0(\theta_{-t-1 } \omega)  )
+ \nabla z_1 (\theta_{s-t-1} \omega) \|^2
 $$
 \be
 \label{p42a_1}
 \le  2 \| \nabla \ut  (s, \theta_{-t-1} \omega, \ut_0(\theta_{-t-1 } \omega)  ) \|^2
+2 \|\nabla z_1 (\theta_{s-t-1} \omega) \|^2 .
\ee
By \eqref{zz} we have
\be
\label{p42a_2}
2 \|\nabla z_1 (\theta_{s-t-1} \omega) \|^2
= 2 \| \nabla \phi_1 \|^2 |y_1 (\theta_{s-t-1} \omega_1) |^2
\le c e^{{\frac \eta 2} ( t+1-s)} r (\omega)
\le c e^{{\frac \eta 2}  } r (\omega).
\ee
 Now  integrating  \eqref{p42a_1}  with respect to $s$
over $(t, t+1)$, by Lemma \ref{lem42}
and inequality \eqref{p42a_2}, we get that
\be
\label{p42a_8}
\int_t^{t+1} \| \nabla u (s, \theta_{-t-1} \omega, u_0(\theta_{-t-1 } \omega) \|^2 ds
\le c_1 + c_2 r  (\omega)  .
\ee
Then the lemma follows  from \eqref{p42a_8}.
\end{proof}

Next, we derive uniform estimates on $u$ in $\hone$.
\begin{lem}
\label{lem43}
Assume that $g, h \in L^2(\R^n)$ and \eqref{f1}-\eqref{f4}
hold.   Let   $ B= \{B(\omega)\}_{\omega \in \Omega}\in \mathcal{D}$.
Then for $P$-a.e. $\omega \in \Omega$, there exists    $T_B(\omega)>0$
such that   for all $t \ge T_B(\omega)$,
$$
  \| \nabla u (t, \theta_{-t } \omega, u_0(\theta_{-t  } \omega) ) \|^2
\le c (1+ r  (\omega) ),
$$
where $c$ is a positive deterministic constant and
  $r(\omega)$ is the  tempered function
 in \eqref{z2}.
\end{lem}

\begin{proof}
Taking the inner product of  \eqref{u1}
  with $\Delta \ut $ in $L^2(\R^n)$, we get that
$$
{\frac 12} {\frac d{dt}} \| \nabla \ut  \|^2
+ \lambda \| \nabla  \ut  \|^2
+ \| \Delta \ut  \|^2
$$
\be
\label{p43_1}
=\alpha (\vt, \Delta \ut)
-\ii f(x, u) \Delta \ut dx
-(g + \Delta \zto -\alpha \ztotwo, \Delta \ut ).
\ee
Note that the first term on the right-hand side of \eqref{p43_1}
is bounded by
\be
\label{p43_1_a1}
\alpha |(\vt, \Delta \ut)|
\le \alpha \| \vt \| \| \Delta \ut \|
\le {\frac 14} \|  \Delta \ut \|^2 + \alpha^2 \| \vt \|^2.
\ee
For  the nonlinear  term in \eqref{p43_1},
by \eqref{f2}-\eqref{f4},  we have
$$
-\ii f(x, u) \;  \Delta \ut dx
= -\ii f(x, u) \;  \Delta u dx +  \ii f(x, u) \; \Delta \zto dx
$$
$$
=\ii {\frac {\partial f}{\partial x}} (x, u) \; \nabla u dx
+ \ii {\frac {\partial f}{\partial u}} (x,u) \;  | \nabla  u |^2 dx
+ \ii f(x, u) \Delta \zto dx
$$
$$
\le \| \psi_3\| \| \nabla u \|
+ \beta \| \nabla u \|^2
+ \ii |f(x,u)| \; |\Delta \zto | dx
$$
$$
\le \| \psi_3\| \| \nabla u \|
+ \beta \| \nabla u \|^2
+  \alpha_2 \ii |u|^{p-1}  \; |\Delta \zto | dx
+
\ii | \psi_2 (x) |  \; |\Delta \zto | dx
$$
$$
\le
c \| \nabla u \|^2
+ {\frac {\alpha_2}{q}} \ii |u|^p dx
+ {\frac {\alpha_2}p}  \ii  \; |\Delta \zto |^p dx
+ c  ( \| \psi_2\|^2 + \| \psi_3 \|^2 ) + c \|   \Delta \zto  \|^2.
$$
\be
\label{p43_2}
\le c \left (
 \| \nabla u \|^2 + \| u \|^p_p \right )
 + c \left (
  \|   \Delta \zto  \|^2 +  \|   \Delta \zto  \|^p_p +1
\right ).
\ee
On the other hand,  the last term on the right-hand side
of \eqref{p43_1} is bounded by
$$
|(g, \Delta \ut)| + |  ( \Delta \zto , \Delta \ut )|
+ \alpha |(\ztotwo, \Delta \ut)|
$$
\be
\label{p43_3}
\le {\frac 14} \| \Delta \ut \|^2
+ c \left (
 \| g \|^2 + \|  \Delta \zto \|^2 + \| \ztotwo \|^2
\right ).
\ee
By \eqref{p43_1}-\eqref{p43_3}  we get that
$$
  {\frac d{dt}} \| \nabla \ut \|^2
+ 2 \lambda \| \nabla  \ut \|^2
+ \| \Delta  \ut \|^2
$$
\be
\label{p43_4}
 \le 2 \alpha^2 \| \vt \|^2 +
 c \left (
 \| \nabla u \|^2 + \| u \|^p_p \right )
 + c \left (
  \|   \Delta \zto  \|^2 +  \|   \Delta \zto  \|^p_p  + \| \ztotwo \|^2
+1
\right ).
\ee
Let
\be
\label{p43_5}
p_2(\theta_t \omega)=  c \left (
  \|   \Delta \zto  \|^2 +  \|   \Delta \zto  \|^p_p  +  \| \ztotwo \|^2
+1
\right ).
\ee
 Since $\zto =\phi_1  y_1(\theta_t \omega_1)$
 and $\ztotwo = \phi_2 y_2 (\theta_t \omega_2)$
 with
 $\phi_1  \in H^2(\R^n) \cap W^{2,p}(\R^n)$
and $\phi_2 \in \hone$,   we find that there are positive
constants $c_1$ and $c_2$ such that
$$
p_2(\theta_{t} \omega )
\le c_1 \sum_{j=1}^2  \left ( |y_j(\theta_t \omega_j)|^2
+ |y_j(\theta_t \omega_j)|^p_p \right ) + c_2,
$$
which  along  with \eqref{zz} shows that
\be
\label{p43_6}
p_2(\theta_{t} \omega )
\le
c_1 e^{{\frac \eta 2} |t| } r (\omega) +c_2, \quad \forall \; t \in \R.
\ee
By \eqref{p43_4}-\eqref{p43_5}, we find that
\be
\label{p43_7}
{\frac d{dt}} \| \nabla \ut \|^2
\le c \left (
 \| \nabla u \|^2 + \| u \|^p_p  + \| \vt \|^2 \right ) + p_2 (\theta_{t} \omega).
 \ee
 Let $T_B(\omega)$ be the positive  constant in Lemma \ref{lem42},
 take $t \ge T_B(\omega)$ and $s \in (t, t+1)$.  Then integrate \eqref{p43_7} over
 $(s, t+1)$    to get
 $$
 \| \nabla \ut(t+1, \omega, \ut_0(\omega)) \|^2
 \le \| \nabla \ut (s, \omega, \ut_0(\omega)) \|^2
 + \int_s^{t+1} p_2(\theta_\tau \omega) d \tau
 $$
 $$
 +
 c \int_s^{t+1}
 \left (
   \| \nabla u(\tau, \omega, u_0(\omega)) \|^2
+ \|  u(\tau, \omega, u_0(\omega)) \|^p_p
+ \| \vt (\tau, \omega, \vt_0(\omega) )\|^2
 \right ) d \tau
 $$
 $$
  \le \| \nabla \ut(s, \omega, \ut_0(\omega)) \|^2
 + \int_t^{t+1} p_2(\theta_\tau \omega) d \tau
 $$
 $$
 +
 c \int_t^{t+1}
 \left (
   \| \nabla u(\tau, \omega, u_0(\omega)) \|^2
+ \|  u(\tau, \omega, u_0(\omega)) \|^p_p
+ \| \vt (\tau, \omega, \vt_0(\omega) )\|^2
 \right ) d \tau .
 $$
 Now integrating the above with respect to $s$ over $(t, t+1)$, we find that
 $$  \| \nabla \ut (t+1, \omega, \ut_0(\omega)) \|^2
  \le \int_t^{t+1}  \| \nabla \ut (s, \omega, \ut_0(\omega)) \|^2 ds
 + \int_t^{t+1} p_2(\theta_\tau \omega) d \tau
 $$
 $$
 +
 c \int_t^{t+1}
 \left (
   \| \nabla u(\tau, \omega, u_0(\omega)) \|^2
+ \|  u(\tau, \omega, u_0(\omega)) \|^p_p
+ \| \vt (\tau, \omega, \vt_0(\omega) )\|^2
 \right ) d \tau .
 $$
 Replacing $\omega$ by $\theta_{-t -1} \omega$, we obtain that
 $$  \| \nabla \ut (t+1, \theta_{-t -1} \omega, \ut_0(\theta_{-t -1} \omega)) \|^2
 $$
 $$
  \le \int_t^{t+1}  \| \nabla \ut (s, \theta_{-t -1} \omega, \ut_0(\theta_{-t -1} \omega)) \|^2 ds
 + \int_t^{t+1} p_2(   \theta_{\tau -t -1} \omega) d \tau
 $$
$$
 +
 c \int_t^{t+1}
 \left (
   \| \nabla u(\tau, \theta_{-t -1}\omega, u_0(\theta_{-t -1}\omega)) \|^2
+ \|  u(\tau, \theta_{-t -1}\omega, u_0(\theta_{-t -1}\omega)) \|^p_p
 \right ) d \tau
$$
\be
\label{p43_10}
+ c \int_t^{t+1}
  \|  \vt (\tau, \theta_{-t -1}\omega, \vt_0(\theta_{-t -1}\omega)) \|^2 d \tau.
\ee
By Lemmas \ref{lem42} and \ref{lem42a}, it follows from
\eqref{p43_10} and \eqref{p43_6}  that, for all $t \ge T_B(\omega)$,
$$  \| \nabla \ut (t+1, \theta_{-t -1} \omega, \ut_0(\theta_{-t -1} \omega)) \|^2
 \le c_3 + c_4  r(\omega)
 + \int_{-1}^0 p_2(\theta_s \omega) ds
 $$
\be
\label{p43_11}
 \le c_3 + c_4  r(\omega)
 +   \int_{-1}^0  \left (c_1 e^{-{\frac \eta 2}s} r(\omega) + c_2 \right )  ds
 \le c_5 + c_6 r (\omega).
\ee
Then by \eqref{p43_11} and \eqref{uv}, we have,
for all $t \ge T_B(\omega)$,
$$
 \| \nabla u  (t+1, \theta_{-t-1} \omega, u_0(\theta_{-t-1 } \omega) \|^2
 =
  \| \nabla \ut (t+1, \theta_{-t-1} \omega, \ut_0(\theta_{-t-1 } \omega)  )
+ \nabla z_1( \omega) \|^2
 $$
$$
 \le  2 \| \nabla \ut (t+1, \theta_{-t-1} \omega, \ut_0(\theta_{-t-1 } \omega)  ) \|^2
+2 \|\nabla z_1 (  \omega) \|^2
\le c_7 + c_8 r(\omega),
$$
 which completes the proof.
\end{proof}

\begin{lem}
\label{lem44}
Assume that $g, h \in L^2(\R^n)$ and \eqref{f1}-\eqref{f4}
hold.   Let   $ B= \{B(\omega)\}_{\omega \in \Omega}\in \mathcal{D}$.
Then for every $\epsilon>0$ and  $P$-a.e. $\omega \in \Omega$,
there exist     $T= T_B(\omega, \epsilon)>0$ and
$R= R (\omega, \epsilon)>0$
such that    for all $t \ge T  $,
$$
 \int_{|x| \ge R } \left (
| \ut(t, \theta_{-t } \omega, \ut_0(\theta_{-t  } \omega) ) |^2
 +
 | \vt(t, \theta_{-t } \omega, \vt_0(\theta_{-t  } \omega) ) |^2
 \right )
 dx \le \epsilon.
$$
\end{lem}

\begin{proof}
Let $\rho$ be a smooth function defined on $   \R^+$ such that
$0\le \rho(s) \le 1$ for all $s \in \R^+$, and
$$
\rho (s) = \left \{
\begin{array}{ll}
  0 & \quad \mbox{for} \ 0\le s \le 1; \\
 1 & \quad \mbox{for}  \  s \ge 2.
\end{array}
\right.
$$
Then there exists a positive  constant
$c$ such that
$ | \rho^\prime (s) | \le c$ for all $s \in \R^+$.

Taking the inner product of \eqref{u1}
with $\beta \rh \ut$ in $L^2(\R^n)$ we find that
$$
{\frac 12} \beta {\frac d{dt}} \ii \rh |\ut|^2 dx
 + \lambda \beta  \ii \rh | \ut |^2 dx - \beta \ii \rh \ut \Delta \ut dx
+ \alpha \beta \ii \rh \ut \vt dx
$$
\be
\label{p44_1_a1}
= \beta \ii f(x, u) \rh \ut dx + \beta \ii \left (
g + \Delta \zto -\alpha \ztotwo
\right ) \rh \ut dx.
 \ee
 Taking the inner product of \eqref{v1} with $\alpha \rh   \vt$ in $L^2(\R^n)$ we find that
$$
{\frac 12} \alpha  {\frac d{dt}} \ii \rh  |\vt |^2 dx  + \alpha \delta  \ii \rh |\vt |^2 dx
-  \alpha \beta \ii \rh \ut \vt dx
$$
 \be
\label{p44_1_a2}
= \alpha \ii \rh h \vt dx  + \alpha \beta \ii\rh
  z_1 (\theta_t \omega )  \vt dx .
\ee
  Adding \eqref{p41_1_a1} and \eqref{p41_1_a2}, we obtain that
  $$
{\frac 12}   {\frac d{dt}} \left ( \beta  \ii \rh |\ut|^2 dx
+ \alpha \ii \rh | \vt |^2 dx \right )
$$
$$
 + \lambda \beta  \ii \rh | \ut |^2 dx
+ \alpha \delta  \ii \rh |\vt |^2 dx
- \beta \ii \rh \ut \Delta \ut dx
$$
$$
= \beta \ii f(x, u) \rh \ut dx + \beta \ii \left (
g + \Delta \zto -\alpha \ztotwo
\right ) \rh \ut dx
$$
\be
\label{p44_1}
+ \alpha \ii \rh h \vt dx  + \alpha \beta \ii\rh
  z_1 (\theta_t \omega )  \vt dx .
 \ee
   We now  estimate the terms in \eqref{p44_1} as follows.
   First we have
  $$
    - \beta \ii  \rh \ut \Delta \ut  dx
    = \beta \ii | \nabla \ut  |^2 \rh dx
    +  \beta \ii \ut \rhp {\frac {2x}{k^2}} \cdot \nabla \ut dx
    $$
     \be
   \label{p44_2}
    = \beta \ii | \nabla \ut |^2 \rh dx
    + \beta  \int_{k\le |x| \le \sqrt{2} k} \ut \rhp {\frac {2x}{k^2}} \cdot \nabla \ut dx.
    \ee
    Note that
    the  second term on the right-hand  side of \eqref{p44_2} is bounded by
    $$
   \beta  |\int_{k\le |x| \le \sqrt{2} k} \ut  \rhp {\frac {2x}{k^2}} \cdot \nabla \ut dx|
    \le {\frac {2 \sqrt{2}}k} \beta  \int _{k \le |x| \le \sqrt{2} k} |\ut| \;
 |\rhp| \;  | \nabla \ut| dx
    $$
    \be
    \label{p44_3}
    \le {\frac ck} \ii |\ut | \; |\nabla \ut | dx
    \le {\frac ck} (\| \ut \|^2 + \| \nabla \ut \|^2 ).
    \ee
    By \eqref{p44_2}-\eqref{p44_3}, we find that
    \be
    \label{p44_4}
        - \beta  \ii   \rh \ut \Delta \ut  dx
        \ge  \beta \ii |\nabla  \ut  |^2 \rh dx
        -    {\frac ck} (\| \ut  \|^2 + \| \nabla  \ut \|^2 ).
        \ee
        For the nonlinear term in \eqref{p44_1},  we have
        \be
\label{p44_5}
  \beta \ii f(x, u) \rh \ut dx
 =  \beta  \ii f(x, u) \rh u dx - \beta  \ii f(x, u) \rh \zto dx.
 \ee
 By \eqref{f1}, the first term on the right-hand side of
 \eqref{p44_5} is bounded by
\be
\label{p44_6}
 \beta \ii f(x, u) \rh u dx
 \le
 -\alpha_1 \beta  \ii |u|^p \rh dx
 + \beta \ii \psi_1 \rh dx.
 \ee
 By \eqref{f2}, the second term on the right-hand side of \eqref{p44_5}
 is bounded by
 $$
\beta |  \ii f(x, u) \rh \zto dx |
$$
$$
\le \alpha_2  \beta \ii |u|^{p-1} \rh |\zto | dx
+ \beta  \ii |\psi_2| \rh | \zto | dx
$$
$$
\le {\frac 12} \alpha_1\beta
\ii |u|^p \rh dx + c \ii |\zto|^p \rh dx
$$
\be
\label{p44_7}
+ {\frac 12}\beta  \ii |\zto|^2 \rh dx
+{\frac 12}\beta  \ii \psi_2^2 \rh dx.
\ee
Then it follows from \eqref{p44_5}-\eqref{p44_7} that
$$
  \beta \ii f(x, u) \rh \ut dx
  \le
 -  {\frac 12} \alpha_1\beta
\ii |u|^p \rh dx + \beta \ii \psi_1 \rh dx
$$
\be
\label{p44_8}
+{\frac 12}\beta  \ii \psi_2^2 \rh dx
 + c \ii  \left (  |\zto|^p +  | \zto|^2 \right ) \rh dx.
\ee
For the second  term on the right-hand side of \eqref{p44_1}, we have that
$$
\beta | \ii (g + \Delta \zto -\alpha \ztotwo ) \rh \ut dx|
$$
\be
\label{p44_9}
\le
{\frac 12} \lambda \beta  \ii \rh |\ut|^2 dx
+ c \ii (g^2 + |\Delta \zto |^2  + | \ztotwo |^2 ) \rh dx.
\ee
For the last two terms  on the right-hand side of
\eqref{p44_1}, we find that
$$
\alpha \ii \rh h \vt dx  + \alpha \beta \ii\rh
  z_1 (\theta_t \omega )  \vt dx
  $$
  \be
  \label{p44_9_a1}
  \le
  {\frac 12} \alpha \delta \ii \rh | \vt |^2 dx
  +c \ii \rh \left ( |\zto|^2 + |h|^2 \right ) dx.
\ee
Finally, by \eqref{p44_1}, \eqref{p44_4} and
\eqref{p44_8}-\eqref{p44_9_a1}, we obtain that
 $$
{\frac 12}   {\frac d{dt}} \left ( \beta  \ii \rh |\ut|^2 dx
+ \alpha \ii \rh | \vt |^2 dx \right )
$$
$$
+ {\frac 12} \lambda \beta \ii \rh |\ut |^2 dx
+ {\frac 12} \alpha \delta \ii \rh | \vt |^2 dx
$$
$$
+ {\frac 12} \alpha_1 \beta \ii \rh |u|^p dx
+ \beta \ii \rh |\nabla \ut |^2 dx
$$
$$
\le {\frac ck} (\| \nabla \ut \|^2 + \| \ut  \|^2)
+ c
\ii
\left ( |\psi_1| +  |\psi_2|^2
+   g^2  + h^2 \right )\rh
dx
$$
\be
\label{p44_10}
+c \ii \left (
|\Delta \zto |^2  + |\zto |^2 + | \zto |^p  + |\ztotwo |^2 \right )
\rh dx.
\ee
Note that \eqref{p44_10} implies that
 $$
   {\frac d{dt}} \left ( \beta  \ii \rh |\ut|^2 dx
+ \alpha \ii \rh | \vt |^2 dx \right )
$$
$$
+  \eta   \left ( \beta  \ii \rh |\ut|^2 dx
+ \alpha \ii \rh | \vt |^2 dx \right )
$$
$$
\le {\frac ck} (\| \nabla \ut \|^2 + \| \ut  \|^2)
+ c
\ii
\left ( |\psi_1| +  |\psi_2|^2
+   g^2  + h^2 \right )\rh
dx
$$
\be
\label{p44_11}
+c \ii \left (
|\Delta \zto |^2  + |\zto |^2 + | \zto |^p  + |\ztotwo |^2 \right )
\rh dx.
\ee
By Lemmas \ref{lem41} and \ref{lem43}, there is
$T_1 = T_1(B,  \omega)>0$ such that for all
$t \ge T_1$,
\be
\label{p44_12}
\| \ut (t, \theta_{-t} \omega, \ut_0(\theta_{-t}\omega) ) \|^2_{H^1(\R^n)}
\le c (1 + r(\omega) ).
\ee
Now integrating \eqref{p44_11} over $(T_1, t)$, we get that, for
all $t \ge T_1$,
$$
  \beta  \ii \rh |\ut (t, \omega, \ut_0(\omega) )|^2 dx
+ \alpha \ii \rh | \vt (t, \omega, \vt_0(\omega) ) |^2 dx
$$
$$
\le e^{\eta(T_1-t)} \left (
  \beta  \ii \rh |\ut (T_1, \omega, \ut_0(\omega) )|^2 dx
+ \alpha \ii \rh | \vt (T_1, \omega, \vt_0(\omega) ) |^2 dx
\right )
$$
$$
 +  {\frac ck}  \int_{T_1}^t
e^{\eta (s-t)}
\left  (\| \nabla \ut (s, \omega, \ut_0(\omega)) \|^2
+ \| \ut (s, \omega, \ut_0(\omega)) \|^2 \right ) ds
$$
$$
+ c \int_{T_1}^t
e^{\eta (s-t)}
\ii
\left (   |\psi_1| +  |\psi_2|^2
+   g^2 + h^2  \right )\rh
dx ds
$$
\be
\label{p44_13}
+c  \int_{T_1}^t
e^{\eta (s-t)}\ii \left (
|\Delta z_1(\theta_s \omega) |^2
+ | z_1(\theta_s \omega) |^2
+ |  z_1(\theta_s \omega) |^p
+ |  z_2(\theta_s \omega) |^2 \right )
\rh dx ds.
\ee
Replacing $\omega$ by $\theta_{-t} \omega$, we obtain from
\eqref{p44_13} that, for all $t \ge  T_1$,
$$
  \beta  \ii \rh |\ut (t, \theta_{-t}\omega, \ut_0(\theta_{-t}\omega) )|^2 dx
+ \alpha \ii \rh | \vt (t, \theta_{-t}\omega, \vt_0(\theta_{-t}\omega) ) |^2 dx
$$
$$
\le\beta
 e^{\eta(T_1-t)}
    \ii \rh |\ut (T_1, \theta_{-t}\omega, \ut_0(\theta_{-t}\omega) )|^2 dx
    $$
    $$
+ \alpha e^{\eta(T_1-t)} \ii \rh |
\vt (T_1, \theta_{-t}\omega, \vt_0(\theta_{-t}\omega) ) |^2 dx
$$
$$
 +  {\frac ck}  \int_{T_1}^t
e^{\eta (s-t)}
\| \nabla \ut (s, \theta_{-t}\omega, \ut_0(\theta_{-t}\omega)) \|^2  ds
$$
$$
  +  {\frac ck}  \int_{T_1}^t
e^{\eta (s-t)} \| \ut (s, \theta_{-t}\omega, \ut_0(\theta_{-t}\omega)) \|^2 ds
$$
$$
+ c \int_{T_1}^t
e^{\eta (s-t)}
\ii
\left (   |\psi_1| +  |\psi_2|^2
+   g^2 + h^2  \right )\rh
dx ds
$$
\be
\label{p44_14}
+c  \int_{T_1}^t
e^{\eta (s-t)}\ii \left (
|\Delta z_1(\theta_{s-t} \omega) |^2
+ | z_1(\theta_{s-t} \omega) |^2
+ |  z_1(\theta_{s-t} \omega) |^p
+ |   z_2(\theta_{s-t} \omega)    |^2 \right )
\rh dx ds.
\ee
In what follows, we estimate the terms in \eqref{p44_14}.
First replacing $t$ by $T_1$ and then replacing
$\omega$ by $\theta_{-t} \omega$ in \eqref{p41_8}, we have
the following bounds for the first  two terms  on the right-hand side of
\eqref{p44_14}.
$$
\beta
 e^{\eta(T_1-t)}
    \ii \rh |\ut (T_1, \theta_{-t}\omega, \ut_0(\theta_{-t}\omega) )|^2 dx
    $$
    $$
+ \alpha e^{\eta(T_1-t)} \ii \rh |
\vt (T_1, \theta_{-t}\omega, \vt_0(\theta_{-t}\omega) ) |^2 dx
$$
$$
\le
 e^{\eta(T_1-t)}  \left (
 e^{-\eta  T_1} \left  ( \beta  \| \ut_0 (\theta_{-t} \omega) \|^2
+  \alpha  \| \vt_0 (\theta_{-t}\omega) \|^2 \right )
+ \int_0^{T_1} e^{\eta(\tau -T_1)}  p_1(\theta_{\tau -t}  \omega ) d\tau
+ c \right )
$$
$$
\le
 e^{-\eta   t } \left  ( \beta  \| \ut_0 (\theta_{-t} \omega) \|^2
+  \alpha  \| \vt_0 (\theta_{-t}\omega) \|^2 \right )
+ \int_0^{T_1} e^{\eta(\tau -t)}  p_1(\theta_{\tau -t}  \omega ) d\tau
+ c  e^{\eta(T_1-t)}
$$
$$
\le
 e^{-\eta   t } \left  ( \beta  \| \ut_0 (\theta_{-t} \omega) \|^2
+  \alpha  \| \vt_0 (\theta_{-t}\omega) \|^2 \right )
+ \int_{-t}^{T_1-t} e^{\eta\tau}  p_1(\theta_{\tau }  \omega ) d\tau
+ c  e^{\eta(T_1-t)}
$$
$$
\le
 e^{-\eta   t } \left  ( \beta  \| \ut_0 (\theta_{-t} \omega) \|^2
+  \alpha  \| \vt_0 (\theta_{-t}\omega) \|^2 \right )
+ \int_{-t}^{T_1-t} e^{{\frac 1 2} \eta \tau}  c_6 r (  \omega ) d\tau
+ c  e^{\eta(T_1-t)}
$$
\be
\label{p44_15}
\le
 e^{-\eta   t } \left  ( \beta  \| \ut_0 (\theta_{-t} \omega) \|^2
+  \alpha  \| \vt_0 (\theta_{-t}\omega) \|^2 \right )
+ {\frac 2\eta} c_6 r(\omega) e^{{\frac 12}\eta (T_1-t)}
+ c  e^{\eta(T_1-t)}
\ee
  where we have used  \eqref{p41_6a1}.
  By \eqref{p44_15}, we find that, given $\epsilon>0$,
 there is $T_2 = (B, \omega, \epsilon) >T_1$ such that
 for all $t \ge T_2$,
 \be
\label{p44_20}
 e^{\eta(T_1-t)}  \left (
    \ii \rh  \left ( \beta |\ut (T_1, \theta_{-t}\omega, \ut_0(\theta_{-t}\omega) )|^2
+      \alpha  |
\vt (T_1, \theta_{-t}\omega, \vt_0(\theta_{-t}\omega) ) |^2  \right ) dx  \right )
\le \epsilon.
\ee
  By Lemma \ref{lem41a},
there is $T_3= T_3(B, \omega)>T_1$ such that
   the  third  term on the right-hand side of \eqref{p44_14}
satisfies, for all $t \ge T_3$,
 $$
  {\frac ck}  \int_{T_1}^t
e^{\eta (s-t)}
\| \nabla \ut  (s, \theta_{-t}\omega, \ut_0(\theta_{-t}\omega)) \|^2  ds
\le {\frac ck} (1+ r(\omega) ).
$$
And hence, there is $R_1 = R_1(\omega, \epsilon)>0$ such that
for all $t \ge T_3$ and $k \ge R_1$,
\be
\label{p44_21}
  {\frac ck}  \int_{T_1}^t
e^{\lambda (s-t)}
\| \nabla v (s, \theta_{-t}\omega, v_0(\theta_{-t}\omega)) \|^2  ds
\le \epsilon .
\ee
First replacing $t$ by $s$ and then replacing
$\omega$ by $\theta_{-t} \omega$ in \eqref{p41_8}, we  find  that
  the  fourth   term on the right-hand side of
\eqref{p44_14} satisfies
    $$
  {\frac ck} \int_{T_1}^t e^{\eta  (s-t)}
\|   \ut (s, \theta_{-t}\omega, \ut_0(\theta_{-t}\omega) ) \|^2 ds
  $$
  $$
  \le
  {\frac ck}  \int_{T_1}^t e^{-\eta t}  \left ( \beta
\| \ut_0(\theta_{-t} \omega)\|^2 + \alpha  \| \vt_0(\theta_{-t} \omega)\|^2 \right ) ds
$$
$$
  + {\frac ck} \int_{T_1}^t
  e^{\eta (s-t)} \int_0^s e^{\eta(\tau-s)} p_1(\theta_{\tau -t} \omega)
   d \tau ds
   + {\frac ck} \int_{T_1}^t e^{\lambda(s-t)} ds
$$
$$
  \le
  {\frac ck}   e^{-\eta t}  (t-T_1)  \left ( \beta
\| \ut_0(\theta_{-t} \omega)\|^2 + \alpha  \| \vt_0(\theta_{-t} \omega)\|^2 \right )
$$
$$
  + {\frac ck}
  + {\frac ck} \int_{T_1}^t
  \int_0^s e^{\eta (\tau-t)} p_1(\theta_{\tau -t} \omega)
   d \tau ds
$$
  $$
  \le
  {\frac ck}   e^{-\eta t}  (t-T_1) \left ( \beta
\| \ut_0(\theta_{-t} \omega)\|^2 + \alpha  \| \vt_0(\theta_{-t} \omega)\|^2 \right )
$$
$$
  + {\frac ck}
  + {\frac ck} \int_{T_1}^t
  \int_{-t}^{s-t} e^{\eta \tau} p_1(\theta_{\tau } \omega)
   d \tau ds
$$
$$
  \le
  {\frac ck}   e^{-\eta t}  (t-T_1) \left ( \beta
\| \ut_0(\theta_{-t} \omega)\|^2 + \alpha  \| \vt_0(\theta_{-t} \omega)\|^2 \right )
$$
$$
  + {\frac ck}
  + {\frac ck}  r(\omega)  \int_{T_1}^t
  \int_{-t}^{s-t} e^{{\frac 12}\eta \tau}   d \tau ds
$$
   $$
  \le
  {\frac ck}   e^{-\eta t}  (t-T_1)  \left ( \beta
\| \ut_0(\theta_{-t} \omega)\|^2 + \alpha  \| \vt_0(\theta_{-t} \omega)\|^2 \right )
  + {\frac ck}
  +   {\frac {4c}{\eta^2 k}}  r(\omega)  ,
$$
which shows that, there is $T_4 = T_4(B, \omega, \epsilon)>T_1 $  and
$R_2 =R_2(\omega, \epsilon)$ such that
for all $t \ge T_4$ and $k \ge R_2$,
\be
\label{p44_22}
  {\frac ck} \int_{T_1}^t e^{\eta (s-t)}
\|   \ut (s, \theta_{-t}\omega, \ut_0(\theta_{-t}\omega) ) \|^2 ds
\le \epsilon.
\ee
Note that $\psi_1 \in L^1(\R^n)$ and  $\psi_2, g, h   \in L^2(\R^n)$.
Therefore, there is $R_3 = R_3(\epsilon)$ such that
for all $k \ge R_3$,
$$
\int_{|x| \ge k}  \left (   |\psi_1| + |\psi_2|^2 +   |g |^2 +  |h|^2  \right ) dx
\le  \epsilon.
$$
Then for the
fifth   term on the right-hand side of
\eqref{p44_14}, we have
$$ c
\int_{T_1}^t e^{\eta (s-t)}
\ii  \left (   |\psi_1| + |\psi_2|^2 + | g|^2  + |h|^2  \right ) \rh  dx ds
$$
\be
\label{p44_23}
\le  c
\int_{T_1}^t e^{\eta (s-t)}
\int_{|x| \ge k}   \left (   |\psi_1| + |\psi_2|^2 +|g|^2 + |h|^2  \right )   dx ds
\le  c \epsilon \int_{T_1}^t e^{\eta (s-t)} ds
\le  c \epsilon,
\ee
where $c$ is independent of $\epsilon$.
Note that
and $\phi_1  \in H^2(\R^n) \cap W^{2,p}(\R^n)$
and $\phi_2 \in \hone$.  Hence
 there is  $R_4 = R_4(\omega, \epsilon)$ such that
 for all $k \ge R_4$,
\be
\label{RR4}
 \int_{|x| \ge k} \left ( |\phi_1 (x)|^2 + |\phi_1 (x)|^p
+ | \Delta \phi_1 (x) |^2  + |\phi_2  (x) |^2 \right ) dx
 \le {\frac \epsilon{r(\omega)}},
 \ee
 where $r(\omega)$ is the tempered function in
\eqref{z2}.
 By \eqref{RR4} and   \eqref{z2}-\eqref{z3},
 we have the following  bounds for the last term on the right-hand
 side of \eqref{p44_14}:
$$
  {c}  \int_{T_1}^t
e^{\eta (s-t)}\ii \left (
|\Delta  z_1 ( \theta_{s-t} \omega )   |^2
+ | z_1 ( \theta_{s-t} \omega )  |^2 + |  z_1( \theta_{s-t} \omega )  |^p
+ |  z_2 ( \theta_{s-t} \omega )  |^2
 \right )
\rh dx ds
$$
$$
\le
 {c}  \int_{T_1}^t
e^{\eta (s-t)}
\int_{|x| \ge k}
 \left (
|\Delta  z_1 ( \theta_{s-t} \omega )   |^2
+ | z_1 ( \theta_{s-t} \omega )  |^2 + |  z_1( \theta_{s-t} \omega )  |^p
+ |  z_2 ( \theta_{s-t} \omega )  |^2
 \right )
  dx ds
$$
$$
 \le  c  \int_{T_1}^t
e^{\eta  (s-t)}
\int_{|x| \ge k} \left (
|\Delta \phi_1 |^2 |  y_1 ( \theta_{s-t} \omega_1 )   |^2
+ | \phi_1 |^2 | y_1( \theta_{s-t} \omega_1 )  |^2
+ | \phi_1 |^p | y_1( \theta_{s-t} \omega _1)  |^p
+ | \phi_2 |^2 | y_2( \theta_{s-t} \omega _2)  |^2
\right )
  dx ds
$$
$$
 \le {\frac {c \epsilon}{ r(\omega)}}   \int_{T_1}^t
e^{\eta (s-t)}\sum_{j=1}^2
 \left ( |  y_j ( \theta_{s-t} \omega_j )   |^2
+ | y_j( \theta_{s-t} \omega _j)  |^p   \right )
    ds
 \le  {\frac {c \epsilon}{  r(\omega)}}    \int_{T_1}^t
e^{\eta (s-t)}  r(\theta_{s-t} \omega )
    ds
$$
\be
\label{p44_24}
\le  {\frac {c \epsilon}{  r(\omega)}}    \int_{T_1-t}^0
e^{\eta \tau}  r(\theta_{\tau} \omega )
    d\tau
    \le  {\frac {c \epsilon}{  r(\omega)}}    \int_{T_1-t}^0
e^{{\frac 12}\eta \tau}  r(   \omega )
    d\tau
    \le c \epsilon.
\ee
Let $T_5=T_5(B, \omega, \epsilon) = \max \{T_1,  T_2, T_3, T_4\}$
and $R_5=R_5(\omega, \epsilon) =\max \{ R_1, R_2, R_3, R_4\}$.
Then it follows from
\eqref{p44_14}, \eqref{p44_20}-\eqref{p44_24} that, for all
$t\ge T_5$ and $k \ge R_5$,
$$
  \beta  \ii \rh |\ut (t, \theta_{-t}\omega, \ut_0(\theta_{-t}\omega) )|^2 dx
+ \alpha \ii \rh | \vt (t, \theta_{-t}\omega, \vt_0(\theta_{-t}\omega) ) |^2 dx
\le c\epsilon,
$$
where $c$ is independent of $\epsilon$, and hence,  for all
$t\ge T_5$ and $k \ge R_5$,
$$
  \int_{|x| \ge \sqrt{2} k}  \left (   \beta
   |\ut (t, \theta_{-t}\omega, \ut_0(\theta_{-t}\omega) )|^2 dx
+ \alpha    | \vt (t, \theta_{-t}\omega, \vt_0(\theta_{-t}\omega) ) |^2
\right ) dx
\le c\epsilon.
$$
This completes the proof.
\end{proof}

We now derive uniform estimates on the tails of
$u$ and $v$ in $\h$.

\begin{lem}
\label{lem45}
Assume that $g , h \in L^2(\R^n)$ and \eqref{f1}-\eqref{f4}
hold.   Let   $ B= \{B(\omega)\}_{\omega \in \Omega}\in \mathcal{D}$.
Then for every $\epsilon>0$ and  $P$-a.e. $\omega \in \Omega$,
there exist     $T = T _B(\omega, \epsilon)>0$ and
$R = R  (\omega, \epsilon)>0$
such that,    for all $t \ge T  $,
$$
 \int_{|x| \ge R }
\left (
 | u  (t, \theta_{-t } \omega,u_0(\theta_{-t  } \omega) )   |^2
  + |  v  (t, \theta_{-t } \omega,v_0(\theta_{-t  } \omega) )   |^2
 \right )
 dx \le \epsilon.
$$
\end{lem}

\begin{proof}
Let $T $ and $R $ be the constants in Lemma \ref{lem44}.
By \eqref{RR4}  and \eqref{z2} we have, for all $t\ge T $ and $k \ge R $,
$$
\int_{|x| \ge R }
\left (   |z_1 (\omega)   |^2  + |z_2 (\omega)   |^2 \right ) dx
=\int_{|x| \ge R }
\left ( |\phi_1 (x) |^2 |y_1 (\omega_1)|^2  + |\phi_2 (x) |^2 |y_2 (\omega_2)|^2
\right ) dx
$$
\be
\label{p45_1}
\le {\frac {\epsilon}{  r(\omega)}}  \left  (  |y_1(\omega_1)|^2
+  |y_2(\omega_2)|^2   \right )
\le   \epsilon  .
\ee
Then by \eqref{p45_1} and Lemma \ref{lem44},
we get that, for all $t\ge T $ and $k \ge R $,
$$
 \int_{|x| \ge R }  \left (
| u (t, \theta_{-t } \omega,u_0(\theta_{-t  } \omega) )   |^2
+ | v (t, \theta_{-t } \omega,v_0(\theta_{-t  } \omega) )   |^2   \right )
 dx
 $$
 $$
 =
 \int_{|x| \ge R }  \left (
| \ut (t, \theta_{-t } \omega, \ut_0(\theta_{-t  } \omega) ) + z_1(\omega)  |^2
+ | \vt (t, \theta_{-t } \omega,\vt_0(\theta_{-t  } \omega) )  + z_2 (\omega)  |^2   \right )
 dx
$$
$$
\le
2  \int_{|x| \ge R }  \left (
| \ut (t, \theta_{-t } \omega, \ut_0(\theta_{-t  } \omega) ) |^2
+ | \vt (t, \theta_{-t } \omega,\vt_0(\theta_{-t  } \omega) ) |^2+ |z_1(\omega)  |^2
 + z_2 (\omega)  |^2   \right )
 dx \le 4 \epsilon,
 $$
 which completes the proof.
\end{proof}

Note that equation \eqref{v1} has no  any smoothing effect on the solutions,
that is, if the initial condition $\vt_0(\omega)$  is in $\h$ only,  then for any
$ t\ge 0$,  $\vt(t, \omega, \vt_0(\omega))$  only belongs
to $\h$, but not $\hone$. Therefore, the compactness of  Sobolev embeddings
cannot be used
directly  to derive the asymptotic compactness of the solution operator.
To overcome the difficulty, we  need to decompose the solution operator
as in the deterministic case.
Let $\vt _1$ and $\vt _2$ be the solutions of the following problems,
respectively,
\be
\label{v1_1}
\left \{
\begin{array}{l}
{\frac {d \vt_1}{dt}} + \delta \vt_1 =0,\vspace{3mm}\\
\vt_1 (0) =\vt_0,\\
\end{array}
\right.
\ee
and
\be
\label{v2_1}
\left \{
\begin{array}{l}
{\frac {d \vt_2}{dt}} + \delta \vt_2 = \beta \ut + h + \beta z_1 (\theta_t \omega)
,\vspace{3mm}\\
\vt_2 (0) = 0 .\\
\end{array}
\right.
\ee
Then we find that
$\vt = \vt_1 + \vt_2$.
Let  $(u_0(\omega), v_0(\omega)) \in B(\omega)$ with
$\{B(\omega)\}_{\omega \in \Omega} \in \mathcal{D}$,  and
  $\vt_0( \omega) = v_0(\omega) -z_2(\omega)$.
Then the solution $\vt_1(t, \omega, \vt_0(\omega))$ of
\eqref{v1_1} satisfies that
$$
\| \vt_1 (t,  \omega, \vt_0(\omega) ) \|^2 =
\| \vt_1 (t,  \omega,  v_0(\omega)  - z_2( \omega )  ) \|^2
=   e^{- \delta t} \|  v_0(\omega)  - z_2( \omega )\|^2,
$$
and
\be
\label{v1_solu}
\| \vt_1 (t, \theta_{-t} \omega,  v_0(\theta_{-t}\omega)  - z_2(\theta_{-t}  \omega ) ) \|^2
=   e^{- \delta t} \|  v_0(\theta_{-t}\omega)  - z_2(\theta_{-t}  \omega )\|^2
 \to 0, \quad \mbox{as} \
t \to \infty.
\ee
On the  other hand, for the solutions of problem \eqref{v2_1}, we have
the following  estimates.

\begin{lem}\label{lemv2}
Assume that $g   \in L^2(\R^n)$, $h \in \hone$,
 and \eqref{f1}-\eqref{f4}
hold.   Let   $ B= \{B(\omega)\}_{\omega \in \Omega}\in \mathcal{D}$.
Then  for  $P$-a.e. $\omega \in \Omega$,
there exists      $T = T _B(\omega)>0$
such that     for all $t \ge T  $,
$$
\| \nabla \vt_2 (t, \theta_{-t} \omega, 0) \|
\le c(1 + r(\omega) ),
$$
where $c$ is a positive deterministic constant and $r(\omega)$
is the tempered function in \eqref{z2}.
\end{lem}

\begin{proof}
 Taking the inner product of \eqref{v2_1} with $\Delta \vt_2$ in $\h$, we get that
\be
\label{pv2_1}
 {\frac 12} {\frac d{dt}} \| \nabla \vt_2 \|^2
 + \delta \| \nabla \vt_2 \|^2 = -\beta (\ut,  \Delta \vt_2)
 - (h, \Delta \vt_2)
 -\beta (\zto, \Delta \vt_2).
\ee
Note that the first term on the right-hand side of \eqref{pv2_1} is bounded
by
\be \label{pv2_2}
\beta |(\ut,  \Delta \vt_2)|
\le \beta \| \nabla \ut \| \| \nabla \vt_2 \|
\le {\frac 18} \delta \| \nabla \vt_2 \|^2 + {\frac {2\beta^2}\delta} \| \nabla \ut \|^2.
\ee
For the second term on the right-hand side of \eqref{pv2_1}, we have
\be\label{pv2_3}
|(h, \Delta \vt_2)|
\le \| \nabla h \| \| \nabla \vt_2 \|
\le {\frac 18} \delta \| \nabla \vt_2 \|^2 + {\frac 2\delta } \| h \|^2_{H^1}.
\ee
The last term on the right-hand side of \eqref{pv2_1} is bounded  by
\be\label{pv2_4}
\beta  \| \nabla \zto \|  \| \nabla \vt_2 \|
\le {\frac 18} \delta \| \nabla \vt_2 \|^2 + {\frac {2\beta^2}\delta} \| \nabla \zto \|^2
\le {\frac 18} \delta \| \nabla \vt_2 \|^2 + c  | y_1 (\theta_t \omega_1)  |^2.
\ee
Then it follows from \eqref{pv2_1}-\eqref{pv2_4} that
$$
   {\frac d{dt}} \| \nabla \vt_2 \|^2
 + \delta \| \nabla \vt_2 \|^2
 \le c + c \| \nabla \ut \|^2 + c |y_1 (\theta_t \omega_1) |^2,
$$
which implies that, for all $t\ge 0$,
\be
\label{pv2_5}
   {\frac d{dt}} \| \nabla \vt_2 \|^2
 + \eta \| \nabla \vt_2 \|^2
 \le c + c \| \nabla \ut \|^2 + c |y_1 (\theta_t \omega_1) |^2.
\ee
Integrating \eqref{pv2_5} on $(0, t)$, we obtain that, for all $t\ge 0$,
$$
\| \nabla \vt_2 (t,   \omega, 0 ) \|^2
\le c \int_0^t e^{\eta (s-t)} ds
+ c \int_0^t e^{\eta (s-t)}
\| \nabla \ut (s,   \omega, \ut_0 (  \omega) )\|^2 ds
+ c \int_0^t e^{\eta (s-t)} | y_1 (\theta_{s } \omega_1 ) |^2 ds
$$
\be\label{pv2_6}
\le {\frac c\eta}
+ c \int_0^t e^{\eta (s-t)}
\| \nabla \ut (s,   \omega, \ut_0 (  \omega) )\|^2 ds
+ c \int_0^t e^{\eta (s-t)} | y_1 (\theta_{s } \omega_1 ) |^2 ds.
\ee
Replacing $\omega$ by $\theta_{-t} \omega$ in \eqref{pv2_6},
by \eqref{zz} and  Lemma \ref{lem41a}   we get that,
for all $t \ge 0$,
$$
\| \nabla \vt_2 (t, \theta_{-t} \omega, 0 ) \|^2
\le   {\frac c\eta}
+ c \int_0^t e^{\eta (s-t)}
\| \nabla \ut (s, \theta_{-t} \omega, \ut_0 (\theta_{-t} \omega) )\|^2 ds
+ c \int_0^t e^{\eta (s-t)} | y_1 (\theta_{s-t} \omega_1 ) |^2 ds
$$
$$
 \le c e^{-\eta t}
 \left ( \| \ut_0 (\theta_{-t} \omega )\|^2 +
  \| \vt_0 (\theta_{-t} \omega )\|^2  \right )
  + c (1 + r(\omega) )
  + c \int_{-t}^0 e^{\eta \tau} | y_1 (\theta_\tau \omega_1) |^2  d \tau.
$$
$$
 \le c e^{-\eta t}
 \left ( \| \ut_0 (\theta_{-t} \omega )\|^2 +
  \| \vt_0 (\theta_{-t} \omega )\|^2  \right )
  + c (1 + r(\omega) )
  + c r(\omega)  \int_{-t}^0 e^{{\frac 12} \eta \tau}   d \tau.
$$
$$
 \le c e^{-\eta t}
 \left ( \| \ut_0 (\theta_{-t} \omega )\|^2 +
  \| \vt_0 (\theta_{-t} \omega )\|^2  \right )
  + c (1 + r(\omega) ),
 $$
 which shows that there is $T_1>0$ such that for all $ t \ge T_1$,
 $$
\| \nabla \vt_2 (t, \theta_{-t} \omega, 0 ) \|^2
\le 2 (1 + r(\omega ) ).
$$
The proof is complete.
\end{proof}

\section{Random attractors}
\setcounter{equation}{0}

In this section, we prove the existence  of a $\mathcal{D}$-random
attractor for the random dynamical system $\Phi$ associated  with
the  stochastic FitzHugh-Nagumo system   on $R^n$. To this end, we
first   establish the
  $\mathcal{D}$-pullback asymptotic compactness of $\Phi$.

 In the sequel,  for every $t \in  \R^+, \omega \in \Omega$
 and $(u_0, v_0) \in \h \times \h$, we denote by
  $$
 \Phi_1 (t, \omega, (u_0, v_0) ) =  \left (0, \   \vt_1(t, \omega, v_0-z_2(\omega) )
\right ) \in \h \times \h ,
 $$
 and
 $$
 \Phi_2 (t, \omega, (u_0, v_0) ) =  \left ( u(t, \omega, u_0 ), \
 \vt_2 (t, \omega, 0) + z_2(\theta_t \omega ) \right ) \in \h \times \h,
 $$
 where $u(t, \omega, u_0 )$ is given by \eqref{uv},
 $ \vt_1(t, \omega, v_0-z_2(\omega) )$ is the solution
 of \eqref{v1_1} with $\vt_1(0) =  v_0-z_2(\omega)$, and
 $ \vt_2 (t, \omega, 0)$ is the solution of
 \eqref{v2_1} with $\vt_2 (0) =0$.
 Then we find that
 $$
 \Phi  (t, \omega, (u_0, v_0) )
= \Phi_1 (t, \omega, (u_0, v_0) )
+
\Phi_2 (t, \omega, (u_0, v_0) )  .
$$

The
  $\mathcal{D}$-pullback asymptotic compactness of $\Phi$
is given by the following lemma.

\begin{lem}
\label{lem51} Assume that $g \in L^2(\R^n)$,
$h \in \hone$
 and
\eqref{f1}-\eqref{f4} hold.  Then the random dynamical system
$\Phi$ is $\mathcal{D}$-pullback asymptotically compact in $\h \times \h$;
that is, for $P$-a.e. $\omega \in \Omega$,
the sequence $\{\Phi(t_n, \theta_{-t_n}\omega,
(u_{0,n} (\theta_{-t_n}  \omega), v_{0,n}(\theta_{-t_n}  \omega) ) \}$
 has a convergent subsequence in $\h \times \h$
 provided
    $t_n \to \infty$,   $ B= \{B(\omega)\}_{\omega \in
\Omega}\in \mathcal{D}$ and $(u_{0,n}(\theta_{-t_n}  \omega),
v_{0,n}(\theta_{-t_n} \omega)  ) \in B(\theta_{-t_n}
\omega)$.
\end{lem}

\begin{proof}
Suppose  that   $t_n \to \infty$,   $ B= \{B(\omega)\}_{\omega \in
\Omega}\in \mathcal{D}$ and
  $(u_{0,n}(\theta_{-t_n}  \omega),
v_{0,n}(\theta_{-t_n} \omega)  ) \in B(\theta_{-t_n}
\omega)$.
Then  by Lemma \ref{lem41} and \eqref{v1_solu},
we find that, for $P$-a.e. $\omega \in \Omega$,
$$ \{ \Phi_2 (t_n, \theta_{-t_n} \omega,
(u_{0,n}(\theta_{-t_n} \omega), v_{0, n}(\theta_{-t_n} \omega)) ) \}_{n=1}^\infty
\quad \mbox{ is bounded in } \   \h \times \h.
$$
Hence, there is $(\xi_1, \xi_2) \in \h \times \h$ such that, up to a subsequence,
\be
\label{p51_1}
\Phi_2 (t_n, \theta_{-t_n} \omega,
(u_{0,n}(\theta_{-t_n} \omega), v_{0, n}(\theta_{-t_n} \omega)) )
 \to (\xi_1,  \xi_2) \quad \mbox{weakly in } \ \h \times \h.
\ee
In what follows, we prove the weak convergence of \eqref{p51_1} is
actually strong convergence. Given $\epsilon>0$
 by \eqref{v1_solu} and Lemmas \ref{lem44}-\ref{lem45},
we find that  there is $T_1 =T_1 (B, \omega, \epsilon)$
 and $R_1 =R_1 (\omega, \epsilon)$ such that for all $t\ge T_1$,
 \be
 \label{p51_2}
 \int_{|x| \ge R_1} \left |
  \Phi_2  (t, \theta_{-t} \omega, ( u_0( \theta_{-t} \omega ),
v_0(\theta_{-t}\omega ) ) ) \right |^2 dx
  \le \epsilon.
  \ee
  Since $t_n \to \infty$, there is $N_1=N_1(B, \omega, \epsilon)$ such
  that $t_n \ge T_1$ for every $n \ge N_1$. Hence it follows from
  \eqref{p51_2} that for all $n \ge N_1$,
  \be
 \label{p51_3}
 \int_{|x| \ge R_1} \left |
  \Phi_2 (t_n, \theta_{-t_n} \omega,
( u_{0,n}(  \theta_{-t_n} \omega ),  v_{0,n}(\theta_{-t_n} \omega ) ) ) \right |^2 dx
  \le \epsilon.
  \ee
  On the other hand, By Lemmas \ref{lem41},  \ref{lem43}
and \ref{lemv2}, there is
  $T_2 =T_2(B, \omega)$ such that for all $t  \ge T_2$,
  \be
 \label{p51_4}
\|
  \Phi_2 (t, \theta_{-t} \omega, ( u_0( \theta_{-t} \omega ), v_0(\theta_{-t} \omega ) ) )
  \|^2 _{H^1(\R^n) \times  H^1(\R^n)}
  \le c(1 + r(\omega) ).
  \ee
  Let $N_2 =N_2(B, \omega)$ be large enough such that
  $t_n \ge T_2$ for $n \ge N_2$. Then by \eqref{p51_4} we find that, for all
  $n \ge N_2$,
  \be
 \label{p51_5}
\|
  \Phi_2  (t_n, \theta_{-t_n} \omega, ( u_{0,n}( \theta_{-t_n}  \omega ),
v_{0,n} (\theta_{-t_n} \omega ) ) )  \|^2 _{H^1(\R^n) \times H^1(\R^n)}
  \le c(1 + r(\omega) ).
  \ee
  Denote by $ Q_{R_1} =\{ x \in \R^n: |x| \le R_1\}$. By the compactness of embedding
  $H^1(Q_{R_1})  \times H^1(Q_{R_1})  \hookrightarrow L^2(Q_{R_1})
\times L^2(Q_{R_1})$,
it follows from \eqref{p51_5} that,  up to a subsequence,
$$
\Phi_2 (t_n, \theta_{-t_n} \omega,
(u_{0,n}(\theta_{-t_n} \omega), v_{0, n}(\theta_{-t_n} \omega)) )
 \to (\xi_1,  \xi_2) \quad \mbox{strongly  in } \ L^2(Q_{R_1}) \times L^2(Q_{R_1}),
$$
which shows that for the given $\epsilon>0$, there exists $N_3=N_3(B, \omega, \epsilon)$
such that for all $n \ge N_3$,
\be
\label{p51_5_a1}
\| \Phi_2 (t_n, \theta_{-t_n} \omega,
(u_{0,n}(\theta_{-t_n} \omega), v_{0, n}(\theta_{-t_n} \omega)) )
 - (\xi_1,  \xi_2) \|_{   L^2(Q_{R_1}) \times L^2(Q_{R_1})}^2 \le \epsilon.
\ee
Note that $(\xi_1, \xi_2) \in \h \times \h$. Therefore there exists
$R_2 =R_2(\epsilon) $ such that
\be
\label{p51_5_a2}
\int_{|x| \ge R_2} \left ( |\xi_1 (x) |^2 + |\xi_2 (x) |^2 \right ) dx
\le \epsilon.
\ee
Let $R_3 =\max \{ R_1, R_2\}$ and $N_4 = \max \{N_1, N_3\}$.
We find that
$$
\| \Phi_2 (t_n, \theta_{-t_n} \omega,
(u_{0,n}(\theta_{-t_n} \omega), v_{0, n}(\theta_{-t_n} \omega)) )
 - (\xi_1,  \xi_2) \|_{  \h \times \h}^2
 $$
 $$
 \le
\int_{|x| \le R_3} |  \Phi_2 (t_n, \theta_{-t_n} \omega,
(u_{0,n}(\theta_{-t_n} \omega), v_{0, n}(\theta_{-t_n} \omega)) )
 - (\xi_1,  \xi_2)  |^2  dx
 $$
 $$
+
\int_{|x| \ge  R_3} |  \Phi_2 (t_n, \theta_{-t_n} \omega,
(u_{0,n}(\theta_{-t_n} \omega), v_{0, n}(\theta_{-t_n} \omega)) )
 - (\xi_1,  \xi_2)  |^2  dx
 $$
 $$
 \le
\int_{|x| \le R_3} |  \Phi_2 (t_n, \theta_{-t_n} \omega,
(u_{0,n}(\theta_{-t_n} \omega), v_{0, n}(\theta_{-t_n} \omega)) )
 - (\xi_1,  \xi_2)  |^2  dx
 $$
 $$
  + 2
 \int_{|x| \ge  R_3} |  \Phi_2 (t_n, \theta_{-t_n} \omega,
(u_{0,n}(\theta_{-t_n} \omega), v_{0, n}(\theta_{-t_n} \omega)) )|^2 dx
+ 2   \int_{|x| \ge  R_3}
  ( |\xi_1 (x) |^2 + |   \xi_2 (x)  |^2  dx
  $$
By
 \eqref{p51_3}, \eqref{p51_5_a1}-\eqref{p51_5_a2}, it follows
from the above that,
for all $n \ge N_4$,
 $$
\| \Phi_2 (t_n, \theta_{-t_n} \omega,
(u_{0,n}(\theta_{-t_n} \omega), v_{0, n}(\theta_{-t_n} \omega)) )
 - (\xi_1,  \xi_2) \|_{  \h \times \h}^2 \le 5 \epsilon,
 $$
 which shows that
 \be
 \label{p51_10}
  \Phi_2 (t_n, \theta_{-t_n} \omega,
(u_{0,n}(\theta_{-t_n} \omega ), v_{0, n}(\theta_{-t_n} \omega ) ) )
 \to  (\xi_1,  \xi_2)
 \quad \mbox{strongly in} \ \h \times \h.
 \ee
On the other hand, we have
$$  \| \Phi  (t_n, \theta_{-t_n} \omega,
(u_{0,n}(\theta_{-t_n} \omega ), v_{0, n}(\theta_{-t_n} \omega ) ) )
  -   (\xi_1,  \xi_2) \|_{\h \times \h}
  $$
$$ \le   \|  \Phi_2  (t_n, \theta_{-t_n} \omega,
(u_{0,n}(\theta_{-t_n} \omega ), v_{0, n}(\theta_{-t_n} \omega ) ) )
  -   (\xi_1,  \xi_2) \|_{\h \times \h}
  $$
  $$
  + \| \Phi_1  (t_n, \theta_{-t_n} \omega,
(u_{0,n}(\theta_{-t_n} \omega ), v_{0, n}(\theta_{-t_n} \omega ) ) )\|_{\h \times \h}
  $$
  $$ \le   \|  \Phi_2  (t_n, \theta_{-t_n} \omega,
(u_{0,n}(\theta_{-t_n} \omega ), v_{0, n}(\theta_{-t_n} \omega ) ) )
  -   (\xi_1,  \xi_2) \|_{\h \times \h}
  $$
  $$
  + \| \vt_1  ( t_n, \theta_{-t_n} \omega,
  v_{0, n}( \theta_{-t_n} \omega )  - z_2(\theta_{-t_n} \omega) )
\|_{\h \times \h}
  $$
  $$ \le   \|  \Phi_2  (t_n, \theta_{-t_n} \omega,
(u_{0,n}(\theta_{-t_n} \omega ), v_{0, n}(\theta_{-t_n} \omega ) ) )
  -   (\xi_1,  \xi_2) \|_{\h \times \h}
  $$
\be
\label{p51_20}
  + e^{-\delta t_n} \| v_{0,n}   ( \theta_{-t_n} \omega )
    - z_2(\theta_{-t_n} \omega) )
\|_{\h }
\ee
Then, by \eqref{v1_solu},  \eqref{p51_10} and \eqref{p51_20}
we find  that
$$  \| \Phi  (t_n, \theta_{-t_n} \omega,
(u_{0,n}(\theta_{-t_n} \omega ), v_{0, n}(\theta_{-t_n} \omega ) ) )
  -   (\xi_1,  \xi_2) \|_{\h \times \h} \to 0,
  $$
  which completes the proof.
\end{proof}

We are now in a  position to present our main result: the
existence of a $\mathcal{D}$-random attractor for $\Phi$ in $\h \times \h$.

\begin{thm}
\label{thm52} Assume that $g \in L^2(\R^n)$, $h \in H^1(\R^n)$
 and
\eqref{f1}-\eqref{f4} hold.  Then the random dynamical system
$\Phi$ has a unique
  $\mathcal{D}$-random attractor in $\h \times \h$.
\end{thm}

\begin{proof}
Notice that $\Phi$ has a closed  absorbing set
$\{K(\omega)\}_{\omega \in \Omega}$ in $\mathcal{D}$ by Lemma
\ref{lem41}, and is $\mathcal{D}$-pullback asymptotically compact
in $\h \times \h$ by Lemma \ref{lem51}. Hence the existence of a unique
$\mathcal{D}$-random attractor for $\Phi$ follows from Proposition
\ref{att}
 immediately.
\end{proof}

 \end{document}